\documentclass[11pt]{article}
\usepackage{amssymb}
\usepackage{amsbsy}
\usepackage[latin1]{inputenc}

\usepackage[dvipsnames,usenames]{xcolor}

\def\v        {{\boldsymbol v}}
\def\u        {{\boldsymbol u}}
\def\w        {{\boldsymbol w}}

\def\g        {{\boldsymbol g}}
\def\f        {{\boldsymbol f}}
\def\h        {{\boldsymbol h}}

\def\chi       {{\boldsymbol \xi}}

\newtheorem{theorem}{Theorem}

\newtheorem{proposition}[theorem]{Proposition}
\newtheorem{corollary}[theorem]{Corollary}
\newenvironment{proof}{{\bf  Proof}:}{\hfill $\square$\newline}

\newtheorem{remark}[theorem]{Remark}
\begin{document}

\title{On the time-dependent grade-two model for the magnetohydrodynamic flow: 2D case}

\author
{I.  Kondrashuk\thanks{Grupo de Matem\'atica Aplicada, Dpto. de Ciencias B\'asicas, Facultad de Ciencias,
Universidad del B\'{\i}o-B\'{\i}o, Campus Fernando May, Casilla 447, Chill\'an, Chile. 
E-mail: {\tt igor.kondrashuk@ubiobio.cl.}  I. K. was supported by Fondecyt (Chile) Grants Nos. 1040368,
1050512 and 1121030, by DIUBB (Chile) Grants Nos. 102609 and 121909 GI/C-UBB.}, E.A. Notte-Cuello\thanks{Dpto de Matem\'aticas, Facultad de Ciencias, Universidad de La Serena, La Serena, Chile. E-mail:
{\tt enotte@userena.cl.} This author's work was partially supported by project DIULS PR14151.}, M. Poblete-Cantellano\thanks{Dpto. de Matem\'atica, Facultad de Ingenier\'{\i}a, Universidad de Atacama, Copiap\'o, Chile. E-mail:
{\tt mpoblete@mat.uda.cl.} This author's work was partially supported by Universidad de Atacama, project DIUDA- 22256.} and M. A. Rojas-Medar\thanks{ Grupo de Matem\'atica
Aplicada, Dpto. de Ciencias B\'asicas, Facultad de Ciencias,
Universidad del B\'{\i}o-B\'{\i}o, Campus Fernando May, Casilla 447,
Chill\'an, Chile. E-mail: {\tt marko@ueubiobio.cl.} This work was
partially supported by project MTM2012-32325, Spain, Grant
1120260, Fondecyt-Chile and 121909 GI/C-UBB.}  }


\maketitle
\begin{abstract}
In this paper we discuss the MHD flow of a second grade fluid, in particular
we prove the existence and uniqueness of a weak solution of a time-dependent
grade two fluid model in a two-dimensional Lipschitz domain. We follow the
methodology of \cite{Girault}, i.e , we use a constructive method which can
be adapted to the numerical analysis of finite-element schemes for solving
this problem numerically.
\end{abstract}



\section{Introduction}
A fluid of grade two is a non-Newtonian fluid of differential type
introduced by Rivlin and Ericksen in \cite{Rivlin}. An analysis in \cite%
{Dunn} shows that the equation of a fluid of grade two is given by%
\begin{eqnarray*}
\frac{\partial }{\partial t}(\u-\alpha \Delta \u)-\nu \Delta 
\u+\sum_{j}\left( \u-\alpha \Delta \u\right)
_{j}\nabla u_{j}-\u\cdot \nabla \left( \u-\alpha \Delta 
\u\right) &=&-\nabla p+\f \\
{\rm div}\,\u &=&0
\end{eqnarray*}%
where $\alpha \geq 0$ is a constant of material, $\nu >0$ is the viscosity
of the fluid, $\u$ is the velocity field, and $p$ is pressure. For $%
\alpha =0$ the cl\'{a}ssical Navier-Stokes equation is obtained.

On the other hand, in several situations the motion of incompressible
electrical conducting fluid can be modeled by the magnetohydrodynamic
equation, which correspond to the Navier-Stokes equations coupled with the
Maxwell equations. In presence of a free motion of heavy ions, not directly
due to the electrical field (see Schluter $[4] $ and Pikelner $%
[3] ),$ \ the MHD equation can be reduced to 
\begin{equation}
\begin{array}{lll}
\displaystyle\frac{\partial \u}{\partial t}-\displaystyle\frac{\nu }{{\rho _{m}}}\Delta 
\u+ \u\cdot \nabla \u-\displaystyle\frac{\mu }{\rho _{m}}\h\cdot \nabla 
\h & = & \f-\displaystyle\frac{1}{\rho _{m}}\nabla ( p^{\ast }+%
\displaystyle\frac{\mu }{2}\h^{2}) \bigskip \\ 
\displaystyle\frac{\partial \h}{\partial t}-\displaystyle\frac{1}{{\mu \sigma }}\Delta 
\h+ \u\cdot \nabla \h-\h\cdot \nabla \u & = & -{\rm grad}\,%
\omega\bigskip \\ 
{\rm div}\,\u={ \rm div}\,\h=0 &  & 
\end{array}
\label{Ch1}
\end{equation}%
with 
\begin{equation}
\u\left\vert _{\partial \Omega }\right. =\h\left\vert
_{\partial \Omega }\right. =0.  \label{LS0}
\end{equation}%
Here, $\u$ and $\h$ are respectively the unknown velocity
and magnetic field; $p^{\ast }$ is the unknown hydrostatic pressure; $\omega$ is
an unknown function related to the heavy ions (in such way that the density
of electric current, $j_{0},$ generated by this motion satisfies the
relation ${\rm rot}\,j_{0}=-\sigma \nabla \omega)$ is the density of mass of
the fluid (assumed to be a positive constant); $\mu >0$ is the constant
magnetic permeability of the medium; $\sigma >0$ is the constant electric
conductivity; $\nu >0$ is the constant viscosity of the fluid; $\f$
is a given external force field.

In the case the MHD equation coupled with the equation of an incompressible
second grade fluid, the model can be write as%
\begin{equation}
\begin{array}{l}
\displaystyle\frac{\partial }{\partial t}(\u-\alpha \Delta \u)-\nu
\Delta \u+{\rm curl}\,\left( \u-\alpha \Delta \u
\right) \times \u-\left( \h\cdot \nabla \right) \h%
= \f-\nabla \left( p^{\ast }+\h^{2}\right) \bigskip \\ 
\displaystyle\frac{\partial \h}{\partial t}-\Delta \h+\left( \u%
\cdot \nabla \right) \h-\left( \h\cdot \nabla \right) 
\u=-{\rm grad}\,\omega\bigskip \\ 
{\rm div}\, \u= {\rm div}\,\h=0%
\end{array}
\label{M1}
\end{equation}%
with%
\begin{equation}
\u\left\vert _{\partial \Omega }\right. =\h\left\vert
_{\partial \Omega }\right. =0.  \label{MM1}
\end{equation}%
Note that when $\alpha =0$ we recover the model (\ref{Ch1}).

One of the first mathematical results for this model type appears in \cite%
{Hamdache}, they prove the existence and uniqueness of solutions for a small
time and global existence of solutions for small initial data in a
conducting domain of $\mathbb{R}^{3},$ based on the iterative scheme where
discretization is performed in the spatial variables. In this paper we
discuss the MHD flow of a second grade fluid, in particular we prove the
existence and uniqueness of a weak solution of a time-dependent grade two
fluid model in a two-dimensional Lipschitz domain, where we follow the
methodology of \cite{Girault}, i.e , we use semi-discretization in time and
the work is in a domain of $\mathbb{R}^{2}.$

\section{Preliminary results}

\subsection{Notation}

Let $\left( k_{1},k_{2}\right) $ denote a pair of non-negative integers, set 
$\left\vert k\right\vert =k_{1}+k_{2}$ and define the partial derivative $%
\partial ^{k}$ by 
\[
\partial ^{k}v=\frac{\partial ^{\left\vert k\right\vert }v}{\partial
x_{1}^{k_{1}}\partial x_{2}^{k_{2}}}. 
\]%
Then, for any non-negative integer $m$ and number $r\geq 1,$ recall the
classical Sobolev space%
\[
W^{m,r}(\Omega) =\left\{ v\in L^{r}(\Omega)
;\partial ^{k}v\in L^{r}(\Omega) \forall \left\vert
k\right\vert \leq m\right\} , 
\]%
equipped with the seminorm%
\[
\left\vert v\right\vert _{W^{m,r}(\Omega) }=\left[
\sum_{\left\vert k\right\vert =m}\int_{\Omega }\left\vert
\partial ^{k}v\right\vert ^{r}dx\right] ^{1/r}, 
\]%
and norm (for which it is a Banach space)%
\[
\left\Vert v\right\Vert _{W^{m,r}(\Omega) }=\left[
\sum_{0\leq \left\vert k\right\vert \leq m}\int_{\Omega
}\left\vert v\right\vert _{W^{k,r}(\Omega) }^{r}\right] ^{1/r}, 
\]%
with the usual extension when $r=\infty .$ When $r=2,$ this space is the
Hilbert space $H^{m}(\Omega) .$ The definitions of these spaces
are extended straightforwardly to vectors, with the same notation, but with
the following modification for the norms in the non-Hilbert case. Let 
$\u=(u_{1},u_{2}) ;$ then we set%
\[
\left\Vert \u\right\Vert _{L^{r}(\Omega) }=\left[
\int_{\Omega }\left\Vert \u(x)
\right\Vert ^{r}\right] ^{1/r}, 
\]%
where $\left\Vert\cdot \right\Vert $ denotes the Euclidean vector
norm.

For functions that vanish on the boundary, we define for any $r\geq 1,$%
\[
W_{0}^{1,r}(\Omega) =\left\{ v\in W^{1,r}(\Omega)
;v\left\vert _{\partial \Omega }=0\right. \right\} 
\]%
and recall Poincar\'{e}'s inequality, there exists a constant $\mathcal{P}$
such that 
\begin{equation}
\forall v\in H_{0}^{1}(\Omega) ,\qquad \left\Vert v\right\Vert
_{L^{r}(\Omega) }\leq \mathcal{P}\left\vert v\right\vert
_{H^{1}(\Omega) }.  \label{N1}
\end{equation}%
More generally, recall the inequalities of Sobolev embeddings in two
dimension, for each $r\in \left[ 2,\infty \right) ,$ there exists a constant 
$S_{r}$ such that%
\begin{equation}
\forall v\in H_{0}^{1}(\Omega) ,\qquad \left\Vert v\right\Vert
_{L^{r}(\Omega) }\leq \mathcal{S}_{r}\left\vert v\right\vert
_{H^{1}(\Omega) }.  \label{N2}
\end{equation}%
The case $r=\infty $ is excluded and is replaced by, for any $r>2$ there
exists a constant $M_{r}$ such that%
\begin{equation}
\forall v\in W_{0}^{1,r}(\Omega) ,\qquad \left\Vert
v\right\Vert _{L^{\infty }(\Omega) }\leq \mathcal{M}%
_{r}\left\vert v\right\vert _{W^{1,r}(\Omega) }.  \label{N3}
\end{equation}%
Owing to (\ref{N1}), we use the seminorm $\left\vert \cdot \right\vert
_{H^{1}(\Omega) }$ as a norm on $H_{0}^{1}(\Omega) 
$ and we use it to define the norm of the dual space $H^{-1}( \Omega) $ of $H_{0}^{1}(\Omega) :$%
\[
\left\Vert f\right\Vert _{H^{-1}(\Omega) }=\sup_{v\in
H_{0}^{1}(\Omega) }\frac{\left\langle f,v\right\rangle }{%
\left\vert v\right\vert _{H^{1}(\Omega) }}. 
\]%
In addition to the $H^{1}$ norm, it will be convenient to define the
following norm with the parameter $\alpha :$%
\[
\left\Vert v\right\Vert _{\alpha }=\left( \left\Vert v\right\Vert
_{L^{2}(\Omega) }^{2}+\alpha \left\vert v\right\vert
_{H^{1}(\Omega) }^{2}\right) ^{1/2}. 
\]%
In the following, we denote by $\Vert \cdot \Vert $ the $L^2$ norm.

We shall also use the standard space for incompressible flow:%
\[
\begin{array}{l}
H\left({\rm div};\Omega \right) =\left\{ v\in L^{2}(\Omega)
^{2};{\rm div}\, v\in L^{2}(\Omega) \right\} \bigskip \\ 
H\left({\rm curl};\Omega \right) =\left\{ v\in L^{2}(\Omega)
^{2};{\rm curl}\,v\in L^{2}(\Omega) \right\} \bigskip \\ 
V=\left\{ v\in H_{0}^{1}(\Omega) ^{2};{\rm div}\, v=0\,{\rm  in}\,
\Omega \right\} \bigskip \\ 
V^{\perp }=\left\{ v\in H_{0}^{1}(\Omega) ^{2};\forall w\in
V,( \nabla v,\nabla w) =0\right\} \bigskip \\ 
L_{0}^{2}(\Omega) =\left\{ v\in L^{2}(\Omega)
;\int_{\Omega }qdx=0\right\}%
\end{array}%
\]%
and the space transport:%
\[
X_{v}=\left\{ f\in L^{2}\left( \Omega \right) ;v\cdot \nabla f\in
L^{2}\left( \Omega \right) \right\} , 
\]%
where $v$ is a given velocity in $H^{1}\left( \Omega \right) ^{2}.$

\subsection{Auxiliary theoretical results}

To analize, we shall use the following results. The first theorem concerns
the divergence operator in any dimension $d.$ Its proof can be found for
instance in Girault and Raviart \cite{Girault2}.

\begin{theorem}
Let $\Omega $ be a bounded Lipschitz-continuous domain of $\mathbb{R}^{d}.$
The divergence operator is an isomorphism from $V^{\perp }$ onto $%
L_{0}^{2}\left( \Omega \right) $ and there exists a constant $\beta >0$ such
that for all $f\in L_{0}^{2}\left( \Omega \right) ,$ there exists a unique $%
v\in V^{\perp }$ satisfying%
\[
{\rm div}\, v=f\quad {\rm in }\,\,\Omega \quad {\rm and}\quad \left\Vert
v\right\Vert _{H^{1}(\Omega) }\leq \frac{1}{\beta }\left\Vert
f\right\Vert . 
\]
\end{theorem}

The second result concerns the regularity of the Stokes operator in two
dimensions, see \cite{Grisvard}.

\begin{theorem}
Let $\Omega $ be a bounded polygon in the plane.

\begin{enumerate}
\item For each $r\in ] 1,4/3[ ,$ the Stokes operator is an
isomorphism from 
$$
\left[ \left( W^{2,r}(\Omega) \right) ^{2}\cap V\right] \times %
\left[ W^{1,r}(\Omega) \cap L_{0}^{2}(\Omega) %
\right] \quad {\rm onto}\quad L^{r}(\Omega) ^{2}, 
$$
i.e. for each $f\in L^{r}(\Omega) ^{2},$ there exists a
constant $C_{r}$ and a unique pair 
\[
\left( \u,p\right) \in \left[ \left( W^{2,r}(\Omega)
\right) ^{2}\cap V\right] \times \left[ W^{1,r}(\Omega) \cap
L_{0}^{2}(\Omega) \right] 
\]
such that%
\[
- v\Delta \u+\nabla p=\f,\qquad {\rm div}\,\u=0\, {\rm
in }\, \Omega, \,\u=0\,{\rm on }\, \partial \Omega , 
\]%
and%
\[
\left\vert \mathbf{u}\right\vert _{W^{2,r}\left( \Omega \right) }+\left\vert
p\right\vert _{W^{1,r}\left( \Omega \right) }\leq C_{r}\left\Vert
f\right\Vert _{L^{r}\left( \Omega \right) }. 
\]

\item If in addition, $\Omega $ is a convex polygon, then the Stokes
operator is an isomorphism from $\left[ \left( H^{2}(\Omega)
\right) ^{2}\cap V\right] \times \left[ H^{1}(\Omega) \cap
L_{0}^{2}(\Omega) \right] $ onto $L_{0}^{2}(\Omega
) ^{2}.$ Furthermore, there exists a real number $r>2,$ depending on
the largest inner angle of $\partial \Omega $ such that for all $t\in [
2,r] ,$ the Stokes operator is an isomorphism from $\left[ \left(
W^{2,t}(\Omega) \right) ^{2}\cap V\right] \times \left[
W^{1,t}(\Omega) \cap L_{0}^{2}(\Omega) \right] $
onto $L^{t}(\Omega) ^{2}.$
\end{enumerate}
\end{theorem}

The next result concerns the unique solvability of the steady transport
equation in any dimension $d$, see \cite{Girault3}.

\begin{theorem}
Let $\Omega $ be a bounded Lipschitz-continuous domain of $\mathbb{R}^{d}$
and let $\u$ be a given velocity in $V.$

\begin{enumerate}

\item For every $f$ in $L^{2}(\Omega) $ and every constant $%
\gamma >0,$ the transport equation%
\[
z+\gamma \u\cdot \nabla z=f\qquad {\rm in } \,\,\Omega , 
\]%
has a unique solution $z\in X_{\u}$ and 
\begin{equation}
\left\Vert z\right\Vert \leq \left\Vert
f\right\Vert   \label{N4}
\end{equation}

\item The following Green's formula holds:%
\begin{equation}
\forall z,\theta \in X_{\u},\qquad \left( \u\cdot \nabla
z,\theta \right) =-\left( \u\cdot \nabla \theta ,z\right) .
\label{N5}
\end{equation}
\end{enumerate}
\end{theorem}

Finally, the last result establishes compact embeddings in space and time.
Its, proof, due to Simon, see \cite{Simon}.

\begin{theorem}[Simon]
Let $X,E,Y$ be three Banach spaces with continuous embeddings: $X\subset
E\subset Y,$ the imbedding of $X$ into $E$ being compact. Then for any
number $q\in [1,\infty] ,$ the space%
\begin{equation}
\left\{ v\in L^{q}(0,T;X) ;\frac{\partial v}{\partial t}\in
L^{1}(0,T;Y) \right\}  \label{N6}
\end{equation}%
is compactly imbedded in $L^{q}(0,T;E) .$
\end{theorem}

\subsection{Formulation of the problem}

Let $[0,T] $ be a time interval for some positive time $T,$ let $%
\Omega $ be an domain in two dimensions, with a Lipschitz-continuous
boundary $\partial \Omega $ and let $\mathbf{n}$ denote the unit normal to $%
\partial \Omega ,$ pointing outside $\Omega .$ Let $\f\in %
L^{2}(0,T;H({\rm curl};\Omega)) ,$ the initial
velocities $\u_{0}, \h_0\in V$ with ${\rm curl}\,(\u_{0}-\alpha
\Delta \u_{0}) \in L^{2}(\Omega) ,$ and
we expect the velocity $\u\in L^{\infty }(0,T;V) $ with $%
\partial \u/\partial t\in L^{2}(0,T;V) $, the magnetic field $\h \in L^\infty(0,t;V)$ with $\partial \h/\partial t \in L^{2}(0,T;V) $, and
the pressures $p, \omega\in L^{2}(0,T;L_{0}^{2}(\Omega)) .$

The system (\ref{M1}) can be rewritten by introducing the auxiliary variable 
$z={\rm curl}\,( \u-\alpha \Delta \u) ,$ as%
\begin{equation}
\begin{array}{l}
\displaystyle\frac{\partial }{\partial t}(\u-\alpha \Delta \u)-\nu
\Delta \u+z\times \u- (\h\cdot \nabla) 
\h=\f-\nabla \left( p^{\ast }+\h^{2}\right) \bigskip \\ 
\displaystyle\frac{\partial \h}{\partial t}-\Delta \h+\left( \u
\cdot \nabla \right) \h-\left( \h\cdot \nabla \right) 
\u=-{\rm grad}\,\omega\bigskip \\ 
{\rm div}\,\u= {\rm div}\,\h=0\bigskip \\ 
\u=\h=0\,\,{\rm  \ on }\,\,\partial \Omega \times (0,T)
\bigskip \\ 
\u(0) =\u_{0};\quad \h(0) =%
\h_{0}\quad {\rm in }\,\,\Omega%
\end{array}
\label{MM2}
\end{equation}%
Taking the $\alpha{\rm curl}$ in the first of the above equations, it can be
written as:%
\begin{equation}
\alpha \frac{\partial z}{\partial t}+\nu z+\alpha \left( \u\cdot
\nabla \right) z=\nu {\rm curl}\u+\alpha {\rm curl}\left( \h%
\cdot \nabla \right) \h+\alpha{\rm curl} \f-\alpha {\rm curl}\nabla \left(
p^{\ast }+\h^{2}\right)  \label{M3}
\end{equation}%
where we have used the fact that ${\rm curl}\left( z\times \u
\right) =\u\cdot \nabla z,$ valid in two dimensions. Considering the
above equation, we can rewrite the system (\ref{MM2}) as follows:%
\begin{equation}
\begin{array}{ll}
\displaystyle\frac{\partial }{\partial t}(\u-\alpha \Delta \u)-\nu
\Delta \u+z\times \u & = \f+\left( \h\cdot \nabla
\right) \h-\nabla \left( p^{\ast }+\h^{2}\right) \bigskip \\ 
\displaystyle\frac{\partial \h}{\partial t}-\Delta \h+\left( \u%
\cdot \nabla \right) \h-\left( \h\cdot \nabla \right) 
\u & =-{\rm grad}\, \omega\bigskip \\ 
\alpha\displaystyle \frac{\partial z}{\partial t}+\nu z+\alpha \left( \u\cdot
\nabla \right) z-\nu {\rm curl}\u & =\alpha {\rm curl}\left( 
\h\cdot \nabla \right) \h+\alpha {\rm curl}\f\bigskip \\ 
& -\alpha {\rm curl}\nabla \left( p^{\ast }+\h^{2}\right) \bigskip
\\ 
{\rm div}\,\u={\rm div}\,\h=0. & 
\end{array}
\label{M4}
\end{equation}

\textbf{Semi-discretization in time}

Let $N>1$ be an integer, define the time step $k$ by%
\[
k=\frac{T}{N} 
\]%
and the subdivision points $t^{n}=nk.$ For each $n\geq 1,$ we approximate $%
\f\left( t^{n}\right) $ by the average defined almost everywhere in $\Omega $
by%
\[
\f^{n}(x) =\frac{1}{k}\int_{t^{n-1}}^{t^{n}}\f(x,s)
ds. 
\]%
We set 
\[
\u^{0}=\u_{0},\,\,\h^{0}=\h_{0}{\rm  \ and \ }%
z^{0}={\rm curl}\left( \u_{0}-\alpha \Delta \u%
_{0}\right) . 
\]%
Then, our semi-discrete problem reads: Find sequences $\left( \u
^{n}\right) _{n\geq 1},$ $\left( \h^{n}\right) _{n\geq 1},$ $\left(
z^{n}\right) _{n\geq 1}$ $,\left( p^{n}\right) _{n\geq 1}$ and $\left(
\omega^{n}\right) _{n\geq 1}$ such that $\u^{n},\h^{n}\in
V,z^{n}\in L^{2}(\Omega) ,$ and $p^{n},\omega^{n}\in L_{0}^{2}(
\Omega) ,$ solution of:%
\begin{equation}
\begin{array}{c}
\displaystyle\frac{1}{k}\left( \u^{n+1}-\u^{n}\right) -\alpha \frac{1}{k%
}\Delta \left( \u^{n+1}-\u^{n}\right) -\nu \Delta \u
^{n+1}+z^{n}\times \u^{n+1}\bigskip \\ 
=\f^{n+1}+\h^{n+1}\cdot \nabla \h^{n+1}-\nabla \left( p^{\ast
n+1}+\left( \h^{n+1}\right) ^{2}\right), \bigskip \\ 
\displaystyle\frac{1}{k}\left( \h^{n+1}-\h^{n}\right) -\Delta \h
^{n+1}+\u^{n+1}\cdot \nabla \h^{n+1}-\h^{n+1}\cdot
\nabla \u^{n+1}=-{\rm grad}\,\omega^{n+1},\bigskip \\ 
\displaystyle\frac{\alpha }{k}\left( z^{n+1}-z^{n}\right) +\nu z^{n+1}+\alpha \u%
^{n+1}\cdot \nabla z^{n+1}=\nu {\rm curl}\u^{n+1}+\alpha {\rm curl}%
\f^{n+1}\bigskip \\ 
+\alpha {\rm curl}\left( \h^{n+1}\cdot \nabla \h
^{n+1}\right) -\alpha {\rm curl}\nabla \left( p^{\ast n+1}+\left( \h
^{n+1}\right) ^{2}\right)%
\end{array}
\label{A1}
\end{equation}%
Now, we will make some estimates for $\u^{i},\h^{i},z^{i},$ $%
p^{i}$ and $\omega^{i}.$ Multiplying the first Eq. of (\ref{A1}) by $2k\u
^{n+1},$ the second Eq. of (\ref{A1}) by $2k\h^{n+1}$ and the third
Eq. of (\ref{A1}) by $2kz^{n+1}$ and observing that  ${\rm curl}\nabla F=0,$ for
any vector field $F,$  we obtain%

\begin{equation}
\begin{array}{l}
2\left( \u^{i+1}-\u^{i},\u^{i+1}\right) -2\alpha
\left( \Delta \left( \u^{i+1}-\u^{i}\right) ,\u%
^{i+1}\right) -2k\nu \left( \Delta \u^{i+1},\u^{i+1}\right)
\bigskip \\ 
=2k\left( \f^{i+1},\u^{i+1}\right) +2k\left( \h^{i+1}\cdot
\nabla \h^{i+1},\u^{i+1}\right), \bigskip \\ 
2\left( \h^{i+1}-\h^{i},\h^{i+1}\right) -2k\left(
\Delta \h^{i+1},\h^{i+1}\right) =2k\left( \h
^{i+1}\cdot \nabla \u^{i+1},\h^{i+1}\right), \bigskip \\ 
2\left( z^{i+1}-z^{i},z^{i+1}\right) +\displaystyle\frac{2\nu k}{{\alpha}}\left(
z^{i+1},z^{i+1}\right) =2k\left( {\rm curl}\f^{i+1},z^{i+1}\right) \bigskip
\\ 
+\displaystyle\frac{2\nu k}{{\alpha}}\left( {\rm curl}\u^{i+1},z^{i+1}\right)
+2k\left( {\rm curl}\left( \h^{i+1}\cdot \nabla \h
^{i+1}\right) ,z^{i+1}\right),
\end{array}
\label{M5}
\end{equation}
where we used that fact that $(\u^{i+1} \cdot\h^{i+1}, \h^{i+1})=0$.

\begin{proposition}
The sequence $(\u^{n}) _{n\geq 1}$ and $(\h%
^{n}) _{n\geq 1}$ satisfy the following uniform a priori estimates:%
\begin{equation}
\begin{array}{c}
\displaystyle\sum_{i=0}^{n-1}k\left\Vert \nabla \u^{i+1}\right\Vert
_{\alpha }^{2}\leq \frac{C^{2}\overline{C}}{\nu ^{2}}\left\Vert \f\right\Vert
_{L^{2}(\Omega ,\times ] 0,t^{n}[ ) }^{2}+\frac{1}{%
\nu }\left\Vert \u_{0}\right\Vert _{\alpha }^{2}+\frac{1}{\nu }%
\left\Vert \h_{0}\right\Vert ^{2},\bigskip \\ 
\displaystyle\sum_{i=0}^{n-1}k\left\Vert \nabla \h^{i+1}\right\Vert
^{2}\leq \frac{C^{2}\overline{C}}{2\nu }%
\left\Vert \f\right\Vert _{L^{2}(\Omega \times ] 0,t^{n}[
) }^{2}+\frac{1}{2}\left\Vert \u_{0}\right\Vert _{\alpha }^{2}+%
\frac{1}{2}\left\Vert \h_{0}\right\Vert ^{2}, \\
2\left\Vert \nabla \mathbf{h}^{i+1}\right\Vert
^{2}+\displaystyle\sum\limits_{j=1}^{i}\left( \left\Vert \nabla \h^{j+1}-\nabla 
\h^{j}\right\Vert ^{2}+\displaystyle\frac{k}{\mu \sigma }\left\Vert A\h
^{j+1}\right\Vert ^{2}\right) \leq \left\Vert \nabla \h
_{0}\right\Vert ^{2}.
\end{array}
\label{M6}
\end{equation}
\end{proposition}

\begin{proof}
Multiplying the first Eq. of (\ref{A1}) by $2\u^{i+1}$ and the
second Eq. of (\ref{A1}) by $2\h^{i+1},$ we obtain%
\[
\begin{array}{l}
\displaystyle\frac{2}{k}\left( \u^{i+1}-\u^{i},\u^{i+1}\right)
-\alpha \frac{2}{k}\Delta \left( \u^{i+1}-\u^{i},\u
^{i+1}\right) -2\nu \left( \Delta \u^{i+1},\u^{i+1}\right)
\bigskip \\ 
=2\left( \f^{i+1},\u^{i+1}\right) +2\left( \h^{i+1}\cdot
\nabla \h^{i+1},\u^{i+1}\right), \bigskip \\ 
\displaystyle\frac{2}{k}\left( \h^{i+1}-\h^{i},\h^{i+1}\right)
-2\left( \Delta \h^{i+1},\h^{i+1}\right) -2\left( \h
^{i+1}\cdot \nabla \u^{i+1},\h^{i+1}\right) =0%
\end{array}%
\]%
where we should note that%
\[
\begin{array}{ll}
\left( z^{i}\times \u^{i+1},\u^{i+1}\right) =0,& \qquad
\left( \nabla \left( p^{\ast i+1}+\left( \h^{i+1}\right) ^{2}\right)
,\u^{i+1}\right) =0,\bigskip \\ 
\left( {\rm grad}\, \omega^{i+1},\h^{i+1}\right) =0, & \qquad \left( \u^{i+1}\cdot \nabla \h^{i+1},\h^{i+1}\right) =0.%
\end{array}%
\]
Using the formula 
\begin{equation}
2( a-b,a) =\left\Vert a\right\Vert ^{2}-\left\Vert b\right\Vert
^{2}+\left\Vert a-b\right\Vert ^{2},  \label{M6/2}
\end{equation}
that is true in any Hilbert space, 
and adding the above equations, adding from $i=0$ to $n-1$ and making use the telescopic property, we have
\begin{equation}
\begin{array}{l}
\displaystyle\frac{1}{k}\left\Vert \u^{n}\right\Vert _{\alpha }^{2}+\frac{1}{k}%
\left\Vert \h^{n}\right\Vert ^{2}+%
\frac{1}{k}\sum_{i=0}^{n-1}\left[ \left\Vert \u^{i+1}-%
\u^{i}\right\Vert _{\alpha }^{2}\right] +\frac{1}{k}%
\sum_{i=0}^{n-1}\left\Vert \h^{i+1}-\h%
^{i}\right\Vert^{2}\bigskip \\ 
+2\nu \displaystyle\sum_{i=0}^{n-1}\left\Vert \nabla \u^{i+1}\right\Vert
^{2}+2\sum_{i=0}^{n-1}\left\Vert
\nabla \h^{i+1}\right\Vert^{2}\bigskip
\\ 
\leq 2C\displaystyle\sum_{i=0}^{n-1}\left\Vert \f^{i+1}\right\Vert \left\Vert \u^{i+1}\right\Vert +\frac{1}{k}\left\Vert \u_{0}\right\Vert _{\alpha }^{2}+%
\frac{1}{k}\left\Vert \h_{0}\right\Vert ^{2}.%
\end{array}
\label{M7}
\end{equation}%
Now taking into account that%
\[
\begin{array}{l}
2C\left\Vert \f^{i+1}\right\Vert \left\Vert 
\u^{i+1}\right\Vert \leq \bigskip \\ 
4C^{2}\displaystyle\frac{\delta }{2}\left\Vert \f^{i+1}\right\Vert ^{2}+\frac{1}{2\delta }\left\Vert \u^{i+1}\right\Vert
^{2}\leq 4C^{2}\frac{\delta }{2}\left\Vert
\f^{i+1}\right\Vert ^{2}+\frac{\overline{C}}{%
2\delta }\left\Vert \nabla \u^{i+1}\right\Vert^{2},%
\end{array}%
\]%
then from equation (\ref{M7}) we can write%
\[
\begin{array}{l}
2\nu \displaystyle\sum_{i=0}^{n-1}\left\Vert \nabla \u^{i+1}\right\Vert
^{2}+2\sum_{i=0}^{n-1}\left\Vert
\nabla \h^{i+1}\right\Vert ^{2}\leq
\bigskip \\ 
\displaystyle\sum_{i=0}^{n-1}\left( 4C^{2}\frac{\delta }{2}\left\Vert
\f^{i+1}\right\Vert ^{2}+\frac{\overline{C}}{%
2\delta }\left\Vert \nabla \u^{i+1}\right\Vert^{2}\right) +\frac{1}{k}\left\Vert \u_{0}\right\Vert
_{\alpha }^{2}+\frac{1}{k}\left\Vert \h_{0}\right\Vert
^{2}%
\end{array}%
\]%
then, putting $\delta =\overline{C}/2\nu $, we obtain%
\[
\begin{array}{l}
\nu \displaystyle\sum_{i=0}^{n-1}\left\Vert \nabla \u^{i+1}\right\Vert
^{2}+2\sum_{i=0}^{n-1}\left\Vert
\nabla \h^{i+1}\right\Vert ^{2}\leq
\displaystyle\frac{C^{2}\overline{C}}{\nu k}\left\Vert \f\right\Vert
_{L^{2}(\Omega ,\times ] 0,t^{n}[ ) }^{2}+\frac{1}{k}%
\left\Vert \u_{0}\right\Vert _{\alpha }^{2}+\frac{1}{k}\left\Vert 
\h_{0}\right\Vert ^{2}.
\end{array}%
\]

On the other hand, to obtain estimates of the $\left\Vert \nabla \h
^{n+1}\right\Vert ^{2},$ we multiply the
second equation in (\ref{A1}) by $2A\h^{i+1},$ then we obtain (after applying
the projection operator $P$)%
\[
\frac{2}{k}\left( \nabla \h^{i+1}-\nabla \h^{i},\nabla 
\h^{i+1}\right) +\frac{2}{\mu \sigma }\left\Vert A\h
^{i+1}\right\Vert ^{2}=-2\left( \u^{i+1}\cdot \nabla \h
^{i+1},A\h^{i+1}\right) +2\left( \h^{i+1}\cdot \nabla 
\u^{i+1},A\h^{i+1}\right) , 
\]%
then bounded each of terms, we have%
\[
\frac{2}{k}\left( \nabla \h^{i+1}-\nabla \h^{i},\nabla 
\h^{i+1}\right) =\frac{1}{k}\left\Vert \nabla \h
^{i+1}\right\Vert ^{2}-\frac{1}{k}\left\Vert \nabla \h
^{i}\right\Vert ^{2}+\frac{1}{k}\left\Vert \nabla \h^{i+1}-\nabla 
\h^{i}\right\Vert ^{2},
\]%
\begin{eqnarray*}
\left\vert 2\left( \u^{i+1}\cdot \nabla \h^{i+1},A\h
^{i+1}\right) \right\vert &\leq &2\left\Vert \u^{i+1}\right\Vert
_{L^{6}}^{2}\left\Vert \nabla \mathbf{h}^{i+1}\right\Vert
_{L^{3}}^{2}\left\Vert A\h^{i+1}\right\Vert ^{2}\bigskip \\
&\leq &2\left\Vert \nabla \u^{i+1}\right\Vert ^{2}\left\Vert \nabla 
\h^{i+1}\right\Vert ^{1/2}\left\Vert A\h^{i+1}\right\Vert
^{3/2}\bigskip \\
&\leq &2C_{\varepsilon _{1}}\left\Vert \nabla \u^{i+1}\right\Vert
^{4}\left\Vert \nabla \h^{i+1}\right\Vert ^{2}+2\varepsilon
_{1}\left\Vert A\h^{i+1}\right\Vert ^{2},
\end{eqnarray*}%
\begin{eqnarray*}
\left\vert 2\left( \h^{i+1}\cdot \nabla \u^{i+1},A\h
^{i+1}\right) \right\vert &\leq &2\left\Vert \h^{i+1}\right\Vert
_{L^{\infty }}\left\Vert \nabla \u^{i+1}\right\Vert \left\Vert A%
\h^{i+1}\right\Vert \bigskip \\
&\leq &2C\left\Vert \nabla \h^{i+1}\right\Vert ^{1/2}\left\Vert A%
\h^{i+1}\right\Vert ^{3/2}\left\Vert \nabla \u
^{i+1}\right\Vert \bigskip \\
&\leq &2CC_{\varepsilon _{2}}\left\Vert \nabla \u^{i+1}\right\Vert
^{4}\left\Vert \nabla \h^{i+1}\right\Vert ^{2}+2C\varepsilon
_{2}\left\Vert A\h^{i+1}\right\Vert ^{2},
\end{eqnarray*}%
where we use the estimate of interpolation $\left\Vert \h
^{i+1}\right\Vert _{L^{\infty }}\leq C\left\Vert \nabla \h
^{i+1}\right\Vert ^{1/2}\left\Vert A\h^{i+1}\right\Vert ^{1/2}.$
From above estimates and taking into account that $\left\Vert \nabla 
\u^{i+1}\right\Vert $ is bounded, we have%
\[
\begin{array}{l}
\left\Vert \nabla \h^{i+1}\right\Vert ^{2}+\left\Vert \nabla \h^{i+1}-\nabla \h^{i}\right\Vert ^{2}+\frac{2k}{\mu \sigma }%
\left\Vert A\h^{i+1}\right\Vert ^{2}\bigskip \\ 
\leq k\left( 2C_{\varepsilon _{1}}+2CC_{\varepsilon _{2}}\right) \overline{C}%
\left\Vert \nabla \h^{i+1}\right\Vert ^{2}+k\left( 2\varepsilon
_{1}+2C\varepsilon _{2}\right) \left\Vert A\h^{i+1}\right\Vert
^{2}+\left\Vert \nabla \h^{i}\right\Vert ^{2}%
\end{array}%
\]%
then, there is $k$ and $\varepsilon $ such that (for a sufficiently large $N$%
)$\ 1-k\left( 2C_{\varepsilon _{1}}+2CC_{\varepsilon _{2}}\right) \overline{C%
}=1/2$ and $2k/\mu \sigma -k\left( 2\varepsilon _{1}+2C\varepsilon
_{2}\right) =k/\mu \sigma ,$ thus, from the above inequality we can write%
\[
2\left\Vert \nabla \h^{i+1}\right\Vert ^{2}+\left\Vert \nabla 
\h^{i+1}-\nabla \h^{i}\right\Vert ^{2}+\frac{k}{\mu \sigma }%
\left\Vert A\h^{i+1}\right\Vert ^{2}\leq \left\Vert \nabla \h
^{i}\right\Vert ^{2}, 
\]%
from which we get (using the lemma 3.14 pg. 131 in \cite{Varnhorn} with $%
\eta =\gamma _{i}=\xi =0$) 
\[
2\left\Vert \nabla \h^{i+1}\right\Vert
^{2}+\sum\limits_{j=1}^{i}\left( \left\Vert \nabla \h^{j+1}-\nabla 
\h^{j}\right\Vert ^{2}+\frac{k}{\mu \sigma }\left\Vert A\h
^{j+1}\right\Vert ^{2}\right) \leq \left\Vert \nabla \h
_{0}\right\Vert ^{2}.
\]

\end{proof}

\begin{proposition}
The sequence $(\u^{n}) _{n\geq 1}$ and $(\h
^{n}) _{n\geq 1}$ satisfy the following uniform a priori estimates, for $1\leq n\leq N$
\begin{equation}
\begin{array}{l}
\left\Vert \u^{n}\right\Vert _{\alpha
}^{2}+\displaystyle\sum_{i=0}^{n-1}\left\Vert \u^{i+1}-\u
^{i}\right\Vert _{\alpha }^{2}\leq  
\displaystyle\frac{C^{2}}{2\nu }\left\Vert \f\right\Vert _{L^{2}(\Omega
\times (0,t^{n})) }^{2}+\left\Vert \u
_{0}\right\Vert _{\alpha }^{2}+\left\Vert \h_{0}\right\Vert
^{2},
\end{array}
\label{A2}
\end{equation}%
\begin{equation}
\begin{array}{l}
\left\Vert \h^{n}\right\Vert ^{2}+\displaystyle\sum_{i=0}^{n-1}\left\Vert \h^{i+1}-\h
^{i}\right\Vert ^{2}\leq  
\displaystyle\frac{C^{2}}{2\nu }\left\Vert \f\right\Vert _{L^{2}(\Omega
\times (0,t^{n})) }^{2}+\left\Vert \u%
_{0}\right\Vert _{\alpha }^{2}+\left\Vert \h_{0}\right\Vert
^{2},%
\end{array}
\label{A3}
\end{equation}
\end{proposition}

\begin{proof}
Estimates (\ref{A2}) and (\ref{A3}) are derived by adding the first and
second equations of (\ref{M5}) and using the formula (\ref{M6/2}),%
\[
\begin{array}{l}
\left\Vert \u^{i+1}\right\Vert ^{2}-\left\Vert \u^{i}\right\Vert ^{2}+\left\Vert \u^{i+1}-\u^{i}\right\Vert ^{2}+\left\Vert \h^{i+1}\right\Vert ^{2}\bigskip \\ 
-\left\Vert \h^{i}\right\Vert ^{2}+\left\Vert \h^{i+1}-\h^{i}\right\Vert ^{2}-2\alpha \left( \Delta \left( \u^{i+1}-\u
^{i}\right) ,\u^{i+1}-\u^{i}\right) \bigskip \\ 
-2\alpha \left( \Delta \left( \u^{i+1}-\u^{i}\right) ,%
\u^{i}\right) -2k\nu \left\Vert \nabla \u^{i+1}\right\Vert
^{2}+2k\left\Vert \nabla \h
^{i+1}\right\Vert^{2}=2k\left( \f^{i+1},%
\u^{i+1}\right)%
\end{array}%
\]%
then adding from $i=0$ to $i=n-1$ and making use of the telescopic property, and again using the formulae (\ref{M6/2}), we have
\begin{equation}
\begin{array}{l}
\left\Vert \u^{n}\right\Vert _{\alpha }^{2}-\left\Vert \u%
_{0}\right\Vert _{\alpha }^{2}+\displaystyle\sum_{i=0}^{n-1}\left\Vert \u
^{i+1}-\u^{i}\right\Vert _{\alpha }^{2}+\left\Vert \h
^{n}\right\Vert ^{2}\bigskip \\ 
-\left\Vert \h_{0}\right\Vert^{2}+\displaystyle\sum_{i=0}^{n-1}\left\Vert \h^{i+1}-\h
^{i}\right\Vert ^{2}+2\alpha
\sum_{i=0}^{n-1}\left\Vert \nabla \left( \u^{i+1}-\u
^{i}\right) \right\Vert ^{2}\bigskip \\ 
+2k\displaystyle\sum_{i=0}^{n-1}\left\Vert \nabla \h^{i+1}\right\Vert
^{2}\leq \frac{C^{2}}{2\nu }%
\sum_{i=0}^{n-1}k\left\Vert \f^{i+1}\right\Vert ^{2},%
\end{array}
\label{A51}
\end{equation}%
now, drooping $2\alpha \displaystyle\sum_{i=0}^{n-1}\left\Vert \nabla \left( 
\u^{i+1}-\u^{i}\right) \right\Vert ^{2}$ and $2k\displaystyle\sum_{i=0}^{n-1}\left\Vert \nabla \h%
^{i+1}\right\Vert ^{2}$ we have%
\[
\begin{array}{l}
\left\Vert \u^{n}\right\Vert _{\alpha
}^{2}+\displaystyle\sum_{i=0}^{n-1}\left\Vert \u^{i+1}-\u
^{i}\right\Vert _{\alpha }^{2}+\left\Vert \h^{n}\right\Vert
^{2}+\sum_{i=0}^{n-1}\left\Vert \h
^{i+1}-\h^{i}\right\Vert ^{2}\bigskip
\\ 
\leq \displaystyle\frac{C^{2}}{2\nu }\left\Vert \f\right\Vert _{L^{2}(
\Omega ,\times ] 0,t^{n}[) }^{2}+\left\Vert \u%
_{0}\right\Vert _{\alpha }^{2}+\left\Vert \h_{0}\right\Vert
^{2},%
\end{array}%
\]%
where $\displaystyle\sum_{i=0}^{n-1}k\left\Vert \f^{i+1}\right\Vert
^{2}=\left\Vert \f\right\Vert
_{L^{2}(\Omega ,\times ] 0,t^{n}[) }^{2}.$ From which we get the result.

\end{proof}

\begin{remark} From   (\ref{A51}) we have%
\begin{equation}
\begin{array}{l}
2\alpha \displaystyle\sum_{i=0}^{n-1}\left\Vert \nabla \left( \u^{i+1}-%
\u^{i}\right) \right\Vert ^{2}+2k\sum_{i=0}^{n-1}\left\Vert \nabla \h^{n}\right\Vert
^{2}\leq  
\displaystyle\frac{C^{2}}{2\nu }\left\vert \f\right\vert _{L^{2}(\Omega ,\times %
] 0,t^{n}[) }^{2}+\left\Vert \u_{0}\right\Vert
_{\alpha }^{2}+\left\Vert \h_{0}\right\Vert^{2}.%
\end{array}
\label{A55}
\end{equation}
\end{remark}

\textbf{Estimate for }$z^{n}$

Before obtaining an estimate for $z^{n}$, we will show the following
corollary

\begin{corollary}
Give the sequence $\left( \h^{n}\right) _{n\geq 1}$ we have the
following uniform a priori estimates%
\[
\left\Vert A\h^{n}\right\Vert \leq C%
{\rm \qquad and\qquad }\left\Vert \nabla \left( \h^{n}\cdot \nabla 
\h^{n}\right) \right\Vert \leq C %
\qquad \forall n\in \mathbb{N}.
\]
\end{corollary}

\begin{proof}
Applying the projector $P$ to the second equation of (\ref{A1}) we
obtain%
\[
A\h^{n+1}=-P\left( \u^{n+1}\cdot \nabla \h
^{n+1}\right) +P\left( \h^{n+1}\cdot \nabla \u^{n+1}\right) -%
\frac{1}{k}P\left( \h^{n+1}-\h^{n}\right) 
\]%
then%
\begin{equation}
\begin{array}{l}
\left\Vert A\h^{n+1}\right\Vert\leq 
\left\Vert \u^{n+1}\cdot \nabla \h^{n+1}\right\Vert
+\left\Vert \h^{n+1}\cdot \nabla 
\u^{n+1}\right\Vert +\displaystyle\frac{1}{k}%
\left\Vert \h^{n+1}-\h^{n}\right\Vert %
\end{array}
\label{A6}
\end{equation}%
thus, each term can be estimated as follows:

\begin{description}
\item[a)] 
\begin{eqnarray*}
\left\Vert \u^{n+1}\cdot \nabla \h^{n+1}\right\Vert
 &\leq &\left\Vert \u^{n+1}\right\Vert
_{L^{6}(\Omega) }\left\Vert \nabla \h^{n+1}\right\Vert
_{L^{3}(\Omega) }\bigskip \\
&\leq &C\left\Vert \nabla \u^{n+1}\right\Vert \left\Vert \nabla \h^{n+1}\right\Vert ^{1/2}\left\Vert A\h^{n+1}\right\Vert ^{1/2}\bigskip \\
&\leq &C\left\Vert A\h^{n+1}\right\Vert ^{1/2}\bigskip \\
&\leq &C_{\varepsilon _{1}}+\varepsilon _{1}\left\Vert A\h
^{n+1}\right\Vert \, {\rm for }\, \varepsilon
_{1}>0\,{\rm small}
\end{eqnarray*}%
here was used $H^{1}\hookrightarrow L^{6}$ and a result of interpolation.

\item[b)] 
\begin{eqnarray*}
\left\Vert \h^{n+1}\cdot \nabla \u^{n+1}\right\Vert
 &\leq &\left\Vert \h^{n+1}\right\Vert
_{L^{\infty }(\Omega) }\left\Vert \nabla \u%
^{n+1}\right\Vert \bigskip \\
&\leq &C\left\Vert \h^{n+1}\right\Vert _{L^{\infty }(\Omega
) }\bigskip \\
&\leq &C\left\Vert \h^{n+1}\right\Vert ^{1/2}\left\Vert A\h^{n+1}\right\Vert ^{1/2}\bigskip \\
&\leq &C\left\Vert A\h^{n+1}\right\Vert ^{1/2}\bigskip \\
&\leq &C_{\varepsilon _{2}}+\varepsilon _{2}\left\Vert A\h%
^{n+1}\right\Vert .
\end{eqnarray*}%
\end{description}

Finally, substituting in (\ref{A6}) we obtain%
\[
\left( 1-\varepsilon _{1}-\varepsilon _{2}\right) \left\Vert A\h
^{n+1}\right\Vert \leq C_{\varepsilon
_{1}}+C_{\varepsilon _{2}}+C 
\]%
and considering $\left( 1-\varepsilon _{1}-\varepsilon _{2}\right) >0$ we
obtain%
\begin{equation}
\left\Vert A\h^{n+1}\right\Vert \leq C,
\label{A7}
\end{equation}%
where the constant $C$ is generic.

Now, taking into account the equation (\ref{A7}) and usual estimates, we have%
\[
\begin{array}{l}
\left\Vert \nabla \left( \h^{n+1}\cdot \nabla \h
^{n+1}\right) \right\Vert \bigskip \\ 
\leq C_{1}\left\Vert \nabla \h^{n+1}\cdot \nabla \h
^{n+1}\right\Vert +C_{2}\left\Vert \h
^{n+1}\cdot \nabla ^{2}\h^{n+1}\right\Vert \bigskip \\ 
\leq C_{3}\left\Vert \nabla \h^{n+1}\right\Vert _{L^{3}(\Omega) }\left\Vert \nabla \h^{n+1}\right\Vert _{L^{6}(\Omega) }+C_{4}\left\Vert \h^{n+1}\right\Vert _{L^{\infty }(
\Omega) }\left\Vert \nabla ^{2}\h^{n+1}\right\Vert
\bigskip \\ 
\leq C_{5}\left\Vert A\h^{n+1}\right\Vert \left\Vert A\h^{n+1}\right\Vert +C_{6}^{1/2}\left\Vert A\h^{n+1}\right\Vert ^{1/2}\left\Vert A\h^{n+1}\right\Vert \bigskip \\ 
\leq C,%
\end{array}%
\]%
another consequence of (\ref{A7}) is $\left\Vert \h^{n}\cdot \nabla 
\h^{n+1}\right\Vert \leq C.$
\end{proof}

\begin{proposition}
The sequence $(z^{n}) _{n\geq 1}$ satisfy the following uniform
a priori estimates%
\begin{equation}
\begin{array}{l}
\left\Vert z^{n}\right\Vert^{2}+\displaystyle\sum_{i=0}^{n-1}\left\Vert z^{i+1}-z^{i}\right\Vert ^{2}\leq \bigskip \\ 
\leq \left( \displaystyle\frac{C^{2}\overline{C}}{\alpha }\left\vert \f%
\right\vert _{L^{2}(\Omega ,\times ] 0,t^{n}[ ) }^{2}+%
\frac{\nu }{2\alpha }\left\Vert \u_{0}\right\Vert _{\alpha }^{2}+%
\frac{1}{\alpha }\left\Vert \h_{0}\right\Vert \right) \bigskip \\ 
+\displaystyle\frac{2\alpha k}{\nu }\left\Vert {\rm  curl}\, \f\right\Vert _{L^{2}\left(
\Omega \times \left( 0,t^{n}\right) \right) }^{2}+\frac{2\alpha }{\nu }%
CT+\left\Vert z_{0}\right\Vert ^{2}.%
\end{array}
\label{E0}
\end{equation}
\end{proposition}

\begin{proof}
Multiplying the third equation of (\ref{A1}) by $\frac{2k}{\alpha }z^{i+1},$ we get
\[
\begin{array}{l}
2\left( z^{i+1}-z^{i},z^{i+1}\right) +\displaystyle\frac{2k\nu }{\alpha }\left(
z^{i+1},z^{i+1}\right) =\frac{2k\nu }{\alpha }\left( {\rm curl}\, \u%
^{i+1},z^{i+1}\right) \bigskip \\ 
+2k\left( {\rm curl}\,\f^{i+1},z^{i+1}\right) +2k\left( {\rm curl}\left( 
\h^{i+1}\cdot \nabla \h^{i+1}\right) ,z^{i+1}\right)%
\end{array}%
\]%
then, using the formula (\ref{M6/2}), we obtain%
\[
\begin{array}{l}
\left\Vert z^{i+1}\right\Vert ^{2}-\left\Vert
z^{i}\right\Vert ^{2}+\left\Vert
z^{i+1}-z^{i}\right\Vert ^{2}+\displaystyle\frac{2k\nu }{%
\alpha }\left\Vert z^{i+1}\right\Vert ^{2}\bigskip \\ 
\leq \displaystyle\frac{2k\nu }{\alpha }\left\Vert \nabla \u^{i+1}\right\Vert
\left\Vert z^{i+1}\right\Vert +2k\left\Vert {\rm curl}\, \f^{i+1}\right\Vert \left\Vert z^{i+1}\right\Vert \bigskip \\ 
+2k\left\Vert \nabla \left( \h^{i+1}\cdot \nabla \h
^{i+1}\right) \right\Vert \left\Vert
z^{i+1}\right\Vert \bigskip \\ 
\leq \displaystyle\frac{2k\nu }{\alpha }\left( C_{\varepsilon _{1}}\left\Vert \nabla 
\u^{i+1}\right\Vert ^{2}+\varepsilon
_{1}^{2}\left\Vert z^{i+1}\right\Vert ^{2}\right) +2kC_{\varepsilon _{2}}\left\Vert {\rm curl}f^{i+1}\right\Vert
^{2}\bigskip \\ 
+2k\varepsilon _{2}\left\Vert z^{i+1}\right\Vert ^{2}+2kC\left\Vert z^{i+1}\right\Vert \bigskip \\ 
\leq \displaystyle\frac{2k\nu }{2\alpha \varepsilon _{1}}\left\Vert \nabla \u
^{i+1}\right\Vert^{2}+\frac{k\nu \varepsilon
_{1}}{\alpha }\left\Vert z^{i+1}\right\Vert ^{2}+2\frac{2k}{2\varepsilon _{2}}\left\Vert {\rm curl}\,\f^{i+1}\right\Vert
^{2}\bigskip \\ 
+k\varepsilon _{2}\left\Vert z^{i+1}\right\Vert^{2}+\displaystyle\frac{k}{\varepsilon _{3}}+\ k\varepsilon _{3}\left\Vert
z^{i+1}\right\Vert ^{2}.%
\end{array}%
\]%
Now adding to $i=0$ to $n-1$ we obtain%
\[
\begin{array}{l}
\displaystyle\sum_{i=0}^{n-1}\left( \left\Vert z^{i+1}\right\Vert ^{2}-\left\Vert z^{i}\right\Vert ^{2}\right) +\sum_{i=0}^{n-1}\left\Vert
z^{i+1}-z^{i}\right\Vert ^{2}\bigskip \\ 
+\left( \displaystyle\frac{2k\nu }{\alpha }-\frac{k\nu }{\alpha }\varepsilon
_{1}-k\varepsilon _{2}-kC\varepsilon _{3}\right)
\displaystyle\sum_{i=0}^{n-1}\left\Vert z^{i+1}\right\Vert ^{2}\bigskip \\ 
\leq \displaystyle\frac{k\nu }{\alpha \varepsilon _{1}}\sum_{i=0}^{n-1}\left%
\Vert \nabla \u^{i+1}\right\Vert ^{2}+\frac{k}{\varepsilon _{2}}%
\sum_{i=0}^{n-1}\left\Vert {\rm curl}\,\f^{i+1}\right\Vert
^{2}+C\frac{T}{\varepsilon _{3}}.%
\end{array}%
\]%
Where the term $\displaystyle\frac{k}{\varepsilon _{3}}\sum_{i=0}^{n-1}C=\frac{k%
}{\varepsilon _{3}}Cn\leq \frac{k}{\varepsilon _{3}}C\frac{T}{k}=C\frac{T%
}{\varepsilon _{3}}.$\ Now considering $\varepsilon _{1}=1$ and $\varepsilon
_{2}=\varepsilon _{3}=\nu /2\alpha $ the above equation can be written as%
\[
\begin{array}{l}
\displaystyle\sum_{i=0}^{n-1}\left( \left\Vert z^{i+1}\right\Vert ^{2}-\left\Vert z^{i}\right\Vert ^{2}\right) +\sum_{i=0}^{n-1}\left\Vert
z^{i+1}-z^{i}\right\Vert ^{2}\bigskip \\ 
\leq \displaystyle\frac{k\nu }{\alpha }\sum_{i=0}^{n-1}\left\Vert \nabla \u^{i+1}\right\Vert ^{2}+\frac{2\alpha k}{\nu 
}\sum_{i=0}^{n-1}\left\Vert {\rm curl}\,\f^{i+1}\right\Vert
^{2}+\frac{2\alpha }{\nu }CT,%
\end{array}%
\]%
consequently,%
\[
\begin{array}{l}
\left\Vert z^{n}\right\Vert ^{2}+\displaystyle\sum_{i=0}^{n-1}\left\Vert z^{i+1}-z^{i}\right\Vert
^{2}\leq \left( \frac{C^{2}\overline{C}}{%
\alpha }\left\Vert \f\right\Vert _{L^{2}(\Omega ,\times ]
0,t^{n}[ ) }^{2}+\frac{\nu }{2\alpha }\left\Vert \u%
_{0}\right\Vert _{\alpha }^{2}+\frac{1}{\alpha }\left\Vert \h%
_{0}\right\Vert \right) \\ 
+\displaystyle\frac{2\alpha k}{\nu }\left\Vert {\rm curl}\,\f\right\Vert
_{L^{2}(\Omega \times (0,t^{n})) }^{2}+\frac{%
2\alpha }{\nu }CT+\left\Vert z_{0}\right\Vert ^{2}.%
\end{array}%
\]
\end{proof}

\begin{proposition}
Let 
\[
C_{Z}=\sup_{0\leq n\leq N-1}\left\Vert z^{n}\right\Vert ^{2} .
\]%
The sequences $\left(( \u^{n+1}-\u^{n})
/k\right) _{n\geq 1}$ and $\left(( \h^{n+1}-\h
^{n}) /k\right) _{n\geq 1},$ satisfy the following uniform a priori
estimates:%
\[
\begin{array}{l}
\displaystyle\sum_{i=0}^{n-1}\frac{1}{2k}\left\Vert \u^{i+1}-\u%
^{i}\right\Vert _{\alpha }^{2}\leq \frac{\left( C_{1}\nu
+C_{z}S_{4}^{2}\right) ^{2}}{2\alpha }\left( \frac{C^{2}\overline{C}}{\nu
^{2}}\left\Vert \f\right\Vert _{L^{2}(\Omega ,\times ]
0,t^{n}[) }^{2}\right. \bigskip \\ 
\left. +\displaystyle\frac{1}{\nu }\left\Vert \u_{0}\right\Vert _{\alpha }^{2}+%
\frac{1}{\nu }\left\Vert \h_{0}\right\Vert ^{2}\right) +C_{2}^{2}\left\Vert \f\right\Vert
_{L^{2}\left( \Omega \times \left] 0,t^{n}\right[ \right) }^{2}+DT,\bigskip
\\ 
\displaystyle\sum_{i=0}^{n-1}\frac{1}{2k}\left\Vert \h^{i+1}-\h%
^{i}\right\Vert^{2}\leq \frac{(
C_{1}\nu +C_{z}S_{4}^{2}) ^{2}}{2\alpha }\left( \frac{C^{2}\overline{C}%
}{\nu ^{2}}\left\Vert \f\right\Vert _{L^{2}(\Omega ,\times %
] 0,t^{n}[) }^{2} \right)\bigskip \\ 
\left. +\displaystyle\frac{1}{\nu }\left\Vert \u_{0}\right\Vert _{\alpha }^{2}+%
\frac{1}{\nu }\left\Vert \h_{0}\right\Vert ^{2}\right) +C_{2}^{2}\left\Vert \f\right\Vert
_{L^{2}(\Omega ,\times ] 0,t^{n}[) }^{2}+DT.%
\end{array}%
\]
\end{proposition}

\begin{proof}
Multiplying the first Eq. (\ref{A1}) by $\left( \u^{i+1}-\u
^{i}\right) $ and the second Eq. of (\ref{A1}) by $\left( \h^{i+1}-%
\h^{i}\right) ,$ we obtain%
\begin{equation}
\begin{array}{l}
\displaystyle\frac{1}{k}\left( \u^{i+1}-\u^{i},\u^{i+1}-\u
^{i}\right) -\frac{\alpha }{k}\Delta \left( \u^{i+1}-\u
^{i},\u^{i+1}-\u^{i}\right) -\nu \left( \Delta \u
^{i+1},\u^{i+1}-\u^{i}\right) \bigskip \\ 
+\left( z^{i}\times \u^{i+1},\u^{i+1}-\u^{i}\right)
=\left( \f^{i+1},\u^{i+1}-\u^{i}\right) +\left( \h
^{i+1}\cdot \nabla \h^{i+1},\u^{i+1}-\u^{i}\right),
\bigskip \\ 
\displaystyle\frac{1}{k}\left( \h^{i+1}-\h^{i},\h^{i+1}-\h%
^{i}\right) -\left( \Delta \h^{i+1},\h^{i+1}-\h
^{i}\right) +\left( \u^{i+1}\cdot \nabla \h^{i+1},\h
^{i+1}-\h^{i}\right) \bigskip \\ 
-\left( \h^{i+1}\cdot \nabla \u^{i+1},\h^{i+1}-%
\h^{i}\right) =0,%
\end{array}
\label{A8}
\end{equation}%
then adding the above equations and using that%
\[
\left\vert \left( z^{i}\times \u^{i+1},\u^{i+1}-\u
^{i}\right) \right\vert \leq C_{z}S_{4}^{2}\left\Vert \u
^{i+1}\right\Vert _{H^{1}(\Omega) }\left\Vert \u^{i+1}-%
\u^{i}\right\Vert _{H^{1}(\Omega) },
\]%
we have%
\[
\begin{array}{l}
\displaystyle\frac{1}{k}\left\Vert \u^{i+1}-\u^{i}\right\Vert _{\alpha
}^{2}+\frac{1}{k}\left\Vert \h^{i+1}-\h^{i}\right\Vert
^{2}\bigskip \\ 
\leq \left( C_{1}\nu +C_{z}S_{4}^{2}\right) ^{2}\displaystyle\frac{\varepsilon _{1}}{%
2\alpha }\left\Vert \u^{i+1}\right\Vert _{H^{1}(\Omega)
}^{2}+\frac{\alpha }{2\varepsilon _{1}}\left\Vert \u^{i+1}-\u
^{i}\right\Vert _{H^{1}(\Omega) }^{2}\bigskip \\ 
+\displaystyle\frac{C_{2}^{2}\varepsilon _{2}}{2}\left\Vert \f^{i+1}\right\Vert
^{2}+\frac{1}{2\varepsilon _{2}}\left\Vert 
\u^{i+1}-\u^{i}\right\Vert ^{2}+\frac{C_{3}^{2}k\varepsilon _{3}}{2}\left\Vert \h^{i+1}\cdot
\nabla \h^{i+1}\right\Vert ^{2}\bigskip
\\ 
+\displaystyle\frac{1}{2k\varepsilon _{3}}\left\Vert \u^{i+1}-\u
^{i}\right\Vert^{2}+C_{4}^{2}k\frac{%
\varepsilon _{4}}{2}\left\Vert A\h^{i+1}\right\Vert ^{2}\bigskip \\ 
+\displaystyle\frac{1}{2k\varepsilon _{4}}\left\Vert \h^{i+1}-\h
^{i}\right\Vert ^{2}+C_{5}k\frac{\varepsilon
_{5}}{2}\left\Vert \u^{i+1}\cdot \nabla \h^{i+1}\right\Vert
^{2}\bigskip \\ 
+\displaystyle\frac{1}{2k\varepsilon _{5}}\left\Vert \h^{i+1}-\h
^{i}\right\Vert ^{2}+\frac{%
C_{6}^{2}k\varepsilon _{6}}{2}\left\Vert \h^{i+1}\cdot \nabla 
\u^{i+1}\right\Vert^{2}+\frac{1}{%
2k\varepsilon _{6}}\left\Vert \h^{i+1}-\h^{i}\right\Vert
^{2},%
\end{array}%
\]%
then put $\varepsilon _{1}=k,\varepsilon _{2}=2k,\varepsilon
_{3}=2,\varepsilon _{4}=2,\varepsilon _{5}=4$ and $\varepsilon _{7}=4,$ of
the above equation can be written%
\begin{equation}
\begin{array}{l}
\displaystyle\frac{1}{2k}\left\Vert \u^{i+1}-\u^{i}\right\Vert _{\alpha
}^{2}+\frac{1}{2k}\left\Vert \h^{i+1}-\h^{i}\right\Vert
\bigskip \\ 
\leq \left( C_{1}\nu +C_{z}S_{4}^{2}\right) ^{2}\displaystyle\frac{k}{2\alpha }%
\left\Vert \u^{i+1}\right\Vert _{H^{1}(\Omega)
}^{2}+C_{2}^{2}k\left\Vert \f^{i+1}\right\Vert +C_{3}k\left\Vert \h^{i+1}\cdot \nabla \h
^{i+1}\right\Vert ^{2}\bigskip \\ 
+C_{4}^{2}k\left\Vert A\h^{i+1}\right\Vert ^{2}+2C_{5}^{2}k\left\Vert \u^{i+1}\cdot \nabla \h
^{i+1}\right\Vert ^{2}+2C_{6}^{2}k\left\Vert 
\h^{i+1}\cdot \nabla \u^{i+1}\right\Vert ^{2}.%
\end{array}%
 \label{notte1}
\end{equation}
On the other hand, recalling the Corollary 7, we have%
\[
\begin{array}{l}
\left\Vert \h^{i+1}\cdot \nabla \h^{i+1}\right\Vert
^{2}\leq d_{1},\,\,\,\left\Vert A\h
^{i+1}\right\Vert ^{2}\leq d_{2}\bigskip \\ 
\left\Vert \u^{i+1}\cdot \nabla \h^{i+1}\right\Vert
^{2}\leq d_{3},\,\,\,\left\Vert \h
^{i+1}\cdot \nabla \u^{i+1}\right\Vert^{2}\leq d_{4}%
\end{array}%
\]%
and summing from $i=0$ to $n-1$ in (\ref{notte1}), we obtain%
\[
\begin{array}{l}
\displaystyle\sum_{i=0}^{n-1}\frac{1}{2k}\left\Vert \u^{i+1}-\u
^{i}\right\Vert _{\alpha }^{2}+\sum_{i=0}^{n-1}\frac{1}{2k}%
\left\Vert \h^{i+1}-\h^{i}\right\Vert \bigskip \\ 
\leq \displaystyle\sum_{i=0}^{n-1}\left( C_{1}\nu +C_{z}S_{4}^{2}\right) ^{2}%
\frac{k}{2\alpha }\left\Vert \u^{i+1}\right\Vert _{H^{1}(
\Omega) }^{2}+C_{2}^{2}\sum_{i=0}^{n-1}k\left\Vert \f
^{i+1}\right\Vert ^{2}\bigskip \\ 
+\displaystyle\sum_{i=0}^{n-1}Dk,%
\end{array}%
\]%
where $D=C_{3}^{2}d_{1}+C_{4}^{2}d_{2}+2C_{5}^{2}d_{3}+2C_{6}^{2}d_{4},$
then observed $n\leq N=T/k$ we can write $\displaystyle\sum_{i=0}^{n-1}Dk=Dnk\leq
DT,$ then from the above inequality we obtain%
\[
\begin{array}{l}
\displaystyle\sum_{i=0}^{n-1}\frac{1}{2k}\left\Vert \u^{i+1}-\u
^{i}\right\Vert _{\alpha }^{2}+\sum_{i=0}^{n-1}\frac{1}{2k}%
\left\Vert \h^{i+1}-\h^{i}\right\Vert \bigskip \\ 
\leq \displaystyle\sum_{i=0}^{n-1}\left( C_{1}\nu +C_{z}S_{4}^{2}\right) ^{2}%
\frac{k}{2\alpha }\left\Vert \u^{i+1}\right\Vert _{H^{1}(
\Omega) }^{2}+C_{2}^{2}\sum_{i=0}^{n-1}k\left\Vert \f%
^{i+1}\right\Vert^{2}+DT.%
\end{array}%
\]%
Indeed, from (\ref{M6}) we obtain%
\begin{equation}
\begin{array}{l}
\displaystyle\sum_{i=0}^{n-1}\frac{1}{2k}\left\Vert \u^{i+1}-\u
^{i}\right\Vert _{\alpha }^{2}+\sum_{i=0}^{n-1}\frac{1}{2k}%
\left\Vert \h^{i+1}-\h^{i}\right\Vert \bigskip \\ 
\leq \displaystyle\frac{\left( C_{1}\nu +C_{z}S_{4}^{2}\right) ^{2}}{\alpha }\left( 
\frac{C^{2}\overline{C}}{\nu ^{2}}\left\Vert \f\right\Vert
_{L^{2}(\Omega ,\times ] 0,t^{n}[ ) }^{2}+\frac{1}{
\nu }\left\Vert \u_{0}\right\Vert _{\alpha }^{2}+\frac{1}{\nu }
\left\Vert \h_{0}\right\Vert ^{2}\right) \bigskip \\ 
+2C_{2}^{2}\left\Vert \f\right\Vert _{L^{2}(\Omega ,\times 
] 0,t^{n} [ ) }^{2}+2DT.
\end{array}
\label{A9}
\end{equation}
From which we get the result.
\end{proof}

\begin{proposition}
The sequence $\left( p^{n}\right) _{n\geq 1}$ and $(\omega^{n})
_{n\geq 1}$ satisfy the following uniform a priori estimates:%
\[
\begin{array}{l}
\displaystyle\sum_{i=0}^{n-1}\left[ k\left\Vert p^{i+1}\right\Vert ^{2}+k\left\Vert \omega^{i+1}\right\Vert ^{2}\right] \leq \bigskip \\ 
\left[ \displaystyle\frac{\left( C_{1}\nu +C_{z}S_{4}^{2}\right) ^{2}L}{\alpha }%
+4C_{z}^{2}S_{4}^{4}\right] \left( \displaystyle\frac{C^{2}\overline{C}}{\nu ^{2}}%
\left\Vert \f\right\Vert _{L^{2}(\Omega ,\times ] 0,t^{n}%
[) }^{2}+\displaystyle\frac{1}{\nu }\left\Vert \u_{0}\right\Vert
_{\alpha }^{2}+\frac{1}{\nu }\left\Vert \h_{0}\right\Vert
^{2}\right) \bigskip \\ 
+\left[ 4C_{6}^{2}C+2C_{2}^{2}L\right] \left\Vert \f\right\Vert
^{2}+[ \overline{D}+2LD] T,\quad
1\leq n\leq N.%
\end{array}%
\]
\end{proposition}

\begin{proof}
Let $v_{1},v_{2}\in V^{\perp }=\left\{ v\in H_{0}^{1}\left( \Omega \right)
;\forall w\in V,\left( \nabla v,\nabla w\right) =0\right\} $ such that $%
{\rm div}\,v_{1}=p^{i+1}$ and ${\rm grad}\,v_{2}=\omega^{i+1}.$ Then by
multiplying the first and second eq.(\ref{A1}) by $v_{1}$ and $v_{2}$
respectively, and adding the results, we have%

\[
\begin{array}{l}
\displaystyle\frac{1}{k}\left( \u^{i+1}-\u^{i},v_{1}\right) +\left(
z^{i}\times \u^{i+1},v_{1}\right) +\left( p^{i+1},p^{i+1}\right)
\bigskip \\ 
-\left( \h^{i+1}\cdot \nabla \h^{i+1},v_{1}\right) +\displaystyle\frac{1%
}{k}\left( \h^{i+1}-\h^{i},v_{2}\right) +\left( \u
^{i+1}\cdot \nabla \h^{i+1},v_{2}\right) \bigskip \\ 
-\left( \h^{i+1}\cdot \nabla \u^{i+1},v_{2}\right) +\left(
\omega^{i+1}, \omega^{i+1}\right) =\left( \f^{i+1},v_{1}\right) .%
\end{array}%
\]%
where, we consider that $v_1,v_2 \in V^{\perp}$.
Thus, we can write
\[
\begin{array}{l}
k\left\Vert p^{i+1}\right\Vert ^{2}+k\left\Vert \omega^{i+1}\right\Vert ^{2}\leq
C_{1}\displaystyle\frac{\delta _{1}}{2}\left\Vert \u^{i+1}-\u
^{i}\right\Vert ^{2}+\displaystyle\frac{1}{2\delta _{1}}
\left\Vert v_{1}\right\Vert^{2}\bigskip \\ 
+kC_{z}^{2}S_{4}^{4}\displaystyle\frac{\delta _{2}}{2}\left\Vert \u
^{i+1}\right\Vert _{H^{1}}^{2}+\frac{k}{2\delta _{2}}\left\Vert
v_{1}\right\Vert _{H^{1}}^{2}+kC_{2}^{2}\frac{\delta _{3}}{2}\left\Vert 
\h^{i+1}\cdot \nabla \h^{i+1}\right\Vert ^{2}\bigskip \\ 
+\displaystyle\frac{k}{2\delta _{3}}\left\Vert v_{1}\right\Vert ^{2}+C_{3}^{2}\frac{\delta _{4}}{2}\left\Vert \h^{i+1}-
\h^{i}\right\Vert ^{2}+\frac{1}{
2\delta _{4}}\left\Vert v_{2}\right\Vert ^{2}\bigskip \\ 
+kC_{4}^{2}\frac{\delta _{5}}{2}\left\Vert \u^{i+1}\cdot \nabla 
\h^{i+1}\right\Vert ^{2}+\displaystyle\frac{k}{
2\delta _{5}}\left\Vert v_{2}\right\Vert ^{2}+kC_{5}^{2}\displaystyle\frac{\delta _{6}}{2}\left\Vert \h^{i+1}\cdot
\nabla \u^{i+1}\right\Vert ^{2}\bigskip
\\ 
+\displaystyle\frac{k}{2\delta _{6}}\left\Vert v_{2}\right\Vert ^{2}+kC_{6}^{2}\frac{\delta _{7}}{2}\left\Vert \f
^{i+1}\right\Vert ^{2}+\frac{k}{2\delta _{7}}
\left\Vert v_{1}\right\Vert ^{2}.
\end{array}%
\]%
Now, we considering that $H^{1}\hookrightarrow L^{2},$ i.e., 
\[
\left\Vert v_{1}\right\Vert \leq C\left\Vert
\nabla v_{1}\right\Vert =C\left\Vert
p^{i+1}\right\Vert  
\]
and 
\[
\left\Vert v_{2}\right\Vert \leq C\left\Vert
\nabla v_{2}\right\Vert =C\left\Vert
w^{i+1}\right\Vert .
\]
From which, we obtain
\[
\begin{array}{l}
k\left\Vert p^{i+1}\right\Vert ^{2}+k\left\Vert \omega^{i+1}\right\Vert ^{2}\leq
C_{1}\displaystyle\frac{\delta _{1}}{2}\left\Vert \u^{i+1}-\u%
^{i}\right\Vert ^{2}+\displaystyle\frac{C}{2\delta _{1}}%
\left\Vert p^{i+1}\right\Vert ^{2}\bigskip \\ 
+kC_{z}^{2}S_{4}^{4}\displaystyle\frac{\delta _{2}}{2}\left\Vert \u
^{i+1}\right\Vert _{H^{1}}^{2}+\frac{k}{2\delta _{2}}\left\Vert
p^{i+1}\right\Vert ^{2}+kC_{2}^{2}\frac{%
\delta _{3}}{2}\left\Vert \h^{i+1}\cdot \nabla \h
^{i+1}\right\Vert ^{2}\bigskip \\ 
+\displaystyle\frac{Ck}{2\delta _{3}}\left\Vert p^{i+1}\right\Vert ^{2}+C_{3}^{2}\frac{\delta _{4}}{2}\left\Vert \h^{i+1}-%
\h^{i}\right\Vert ^{2}+\frac{C}{%
2\delta _{4}}\left\Vert \omega^{i+1}\right\Vert ^{2}\bigskip \\ 
+kC_{4}^{2}\displaystyle\frac{\delta _{5}}{2}\left\Vert \u^{i+1}\cdot \nabla 
\h^{i+1}\right\Vert ^{2}+\frac{Ck}{%
2\delta _{5}}\left\Vert \omega^{i+1}\right\Vert ^{2}+kC_{5}^{2}\frac{\delta _{6}}{2}\left\Vert \h^{i+1}\cdot
\nabla \u^{i+1}\right\Vert ^{2}\bigskip
\\ 
+\displaystyle\frac{Ck}{2\delta _{6}}\left\Vert \omega^{i+1}\right\Vert ^{2}+kC_{6}^{2}\frac{\delta _{7}}{2}\left\Vert \f
^{i+1}\right\Vert^{2}+\frac{Ck}{2\delta _{7}}%
\left\Vert p^{i+1}\right\Vert ^{2},%
\end{array}%
\]%
then, taking $\delta _{1}=4C/k,$ $\delta _{2}=4,$ $\delta _{3}=\delta
_{7}=4C $ and $\delta _{4}=2C/k,$ $\delta _{5}=\delta _{6}=4C$,  
 addying from $i=0$ to $n-1$ and multiplying by $2,$ we have%
\[
\begin{array}{l}
\displaystyle\sum_{i=0}^{n-1}\left[ k\left\Vert p^{i+1}\right\Vert ^{2}+k\left\Vert w^{i+1}\right\Vert {2}\right] \leq 4C_{1}^{2}C\sum_{i=0}^{n-1}\frac{1}{k}%
\left\Vert \u^{i+1}-\u^{i}\right\Vert ^{2}\bigskip \\ 
+4C_{z}^{2}S_{4}^{4}\displaystyle\sum_{i=0}^{n-1}k\left\Vert \u
^{i+1}\right\Vert _{H^{1}}^{2}+2C_{3}^{2}C\sum_{i=0}^{n-1}\frac{1}{k%
}\left\Vert \h^{i+1}-\h^{i}\right\Vert ^{2}+4C_{6}^{2}C\left\vert \f\right\vert _{L^{2}(
\Omega ,] 0,T[) }^{2}\bigskip \\ 
+\displaystyle\sum_{i=0}^{n-1}\left(
4C_{2}^{2}Cd_{1}+4C_{4}^{2}Cd_{3}+4C_{5}^{2}Cd_{4}\right) k,%
\end{array}%
\]%
thus, put $\overline{D}=4C_{2}^{2}Cd_{1}+4C_{4}^{2}Cd_{3}+4C_{5}^{2}Cd_{4}$,
noting that $n\leq T/k$ and using the inequalities (\ref{A9}) and (\ref{M6})
, we have%
\[
\begin{array}{l}
\displaystyle\sum_{i=0}^{n-1}\left[ k\left\Vert p^{i+1}\right\Vert^{2}+k\left\Vert \omega^{i+1}\right\Vert ^{2}\right] \leq \bigskip \\ 
\leq \left[ \displaystyle\frac{\left( C_{1}\nu +C_{z}S_{4}^{2}\right) ^{2}L}{\alpha }%
+4C_{z}^{2}S_{4}^{4}\right] \left( \displaystyle\frac{C^{2}\overline{C}}{\nu ^{2}}%
\left\vert \f\right\vert _{L^{2}(\Omega \times ] 0,t^{n}%
[ ) }^{2}+\displaystyle\frac{1}{\nu }\left\Vert \u_{0}\right\Vert
_{\alpha }^{2}+\frac{1}{\nu }\left\Vert \h_{0}\right\Vert
^{2}\right) \bigskip \\ 
+\left[ 4C_{6}^{2}C+2C_{2}^{2}L\right] \left\vert \f\right\vert
_{L^{2}(\Omega \times] 0,T[ ) }^{2}+\left[ \overline{D}+2LD%
\right] T.%
\end{array}%
\]%
where $L=\max \left\{ 4C_{1}^{2}C,4C_{1}^{2}C\right\} $
\end{proof}

\subsection{Existence of solutions}

Here, it is convenient to transform the sequence $(\u
^{n}) ,(\h^{n}) ,( p^{n}) ,(
\omega^{n}) $ and $( z^{n}) $ into functions. Since $( 
\u^{n}) ,(\h^{n}) $ and $(
z^{n}) $ need to be differentiated, we define the piecewise linear
functions in time:%
\[
\begin{array}{c}
\forall t\in [ t^{n},t^{n+1}] ,\quad \u_{k}(
t) =\u^{n}+\displaystyle\frac{t-t^{n}}{k}\left( \u^{n+1}-\u
^{n}\right) ,\quad 0\leq n\leq N-1\bigskip \\ 
\forall t\in [ t^{n},t^{n+1}] ,\quad \h_{k}(
t) =\h^{n}+\displaystyle\frac{t-t^{n}}{k}\left( \h^{n+1}-\h
^{n}\right) ,\quad 0\leq n\leq N-1\bigskip \\ 
\forall t\in [ t^{n},t^{n+1}] ,\quad z_{k}(t) =z^{n}+%
\displaystyle\frac{t-t^{n}}{k}\left( z^{n+1}-z^{n}\right) ,\quad 0\leq n\leq N-1.%
\end{array}%
\]%
Next, in view of the other terms in (\ref{A1}), we define the step functions:%
\[
\begin{array}{l}
\forall t\in \left[ t^{n},t^{n+1}\right] \quad {\rm and}\quad %
0\leq n\leq N-1;\bigskip \\ 
\f_{k}(t) =\f^{n+1},\quad \w_{k}(
t) =\u^{n+1},\quad \g_{k}(t) =\h
^{n+1},\quad p_{k}(t) =p^{n+1},\bigskip \\ 
\omega_{k}(t) =\omega^{n+1},\quad \zeta _{k}(t) =z^{n+1},\quad
\lambda _{k}(t) =z^{n}.%
\end{array}%
\]%
Then we have the following convergences.

\begin{proposition}
The exist functions $\u, \h\in L^{\infty }(
0,T;V) $ with $\partial \u/\partial t,\partial \h
/\partial t\in L^{2}(0,T;V) ,p,\omega\in L^{2}(
0,T;L_{0}^{2}(\Omega)) $ and $z\in L^{\infty }(
0,T;L^{2}(\Omega)) $ such that a subsequence of $k,$
still denoted by $k,$ satisfies:%
\[
\begin{array}{l}
\lim\limits_{k\rightarrow 0}\u_{k}=\lim\limits_{k\rightarrow 0}%
\w_{k}=\u\qquad {\rm weakly* in \,\,}L^{\infty }(
0,T;V) ,\bigskip \\ 
\lim\limits_{k\rightarrow 0}\h_{k}=\lim\limits_{k\rightarrow 0}%
\g_{k}=\h\qquad {\rm weakly* in \,\,}L^{\infty }(
0,T;V) ,\bigskip \\ 
\lim\limits_{k\rightarrow 0}z_{k}=\lim\limits_{k\rightarrow 0}\zeta
_{k}=\lim\limits_{k\rightarrow 0}\lambda _{k}=z\qquad {\rm weakly* in \,\,}%
L^{\infty }(0,T;L^{2}(\Omega)) ,\bigskip \\ 
\lim\limits_{k\rightarrow 0}p_{k}=p\qquad {\rm weakly \,\, in \,\,}L^{2}(
0,T;L_{0}^{2}(\Omega)) ,\bigskip \\ 
\lim\limits_{k\rightarrow 0}\omega_{k}=\omega\qquad {\rm weakly \,\, in \,\,}L^{2}(
0,T;L_{0}^{2}(\Omega)) ,\bigskip \\ 
\lim\limits_{k\rightarrow 0}\displaystyle\frac{\partial }{\partial t}\u_{k}=%
\displaystyle\frac{\partial }{\partial t}\u\qquad {\rm weakly \,\, in \,\,}L^{2}(
0,T;V) ,\bigskip \\ 
\lim\limits_{k\rightarrow 0}\displaystyle\frac{\partial }{\partial t}\h_{k}=%
\frac{\partial }{\partial t}\h\qquad {\rm weakly \,\, in \,\,}L^{2}(
0,T;V) .%
\end{array}%
\]%
Furthermore,%
\begin{equation}
\begin{array}{c}
\lim\limits_{k\rightarrow 0}\u_{k}=\lim\limits_{k\rightarrow 0}%
\w_{k}=\u\qquad {\rm strongly \,\, in \,\,}L^{\infty }(
0,T;L^{4}(\Omega) ^{2}) ,\bigskip \\ 
\lim\limits_{k\rightarrow 0}\h_{k}=\lim\limits_{k\rightarrow 0}%
\g_{k}=\h\qquad {\rm strongly \,\, in \,\,}L^{\infty }(
0,T;L^{4}(\Omega) ^{2})%
\end{array}
\label{E1}
\end{equation}
\end{proposition}

\begin{proof}
Due to the uniform estimates given in Propositions 5 -10, we can extract a
subsequence (still denoted by $k$) such that:%
\[
\begin{array}{l}
\lim\limits_{k\rightarrow 0}\u_{k}=\u;\quad %
\lim\limits_{k\rightarrow 0}\h_{k}=\h\qquad {\rm weakly* in 
\,\,}L^{\infty }(0,T;V) ,\bigskip \\ 
\lim\limits_{k\rightarrow 0}z_{k}=z\qquad {\rm weakly* in \,\,}L^{\infty
}(0,T;L^{2}(\Omega)) ,\bigskip \\ 
\lim\limits_{k\rightarrow 0}p_{k}=p ;\quad \lim\limits_{k\rightarrow
0}w_{k}=w\qquad {\rm weakly \,\, in \,\,}L^{2}(0,T;L_{0}^{2}(\Omega
)) ,\bigskip \\ 
\lim\limits_{k\rightarrow 0}\displaystyle\frac{\partial }{\partial t}\u_{k}=%
\displaystyle\frac{\partial }{\partial t}\u;\quad \lim\limits_{k\rightarrow 0}%
\displaystyle\frac{\partial }{\partial t}\h_{k}=\frac{\partial }{\partial t}%
\h\qquad {\rm weakly \,\, in \,\,}L^{2}(0,T;V) ,\bigskip \\ 
\lim\limits_{k\rightarrow 0}\w_{k}=\w;\quad %
\lim\limits_{k\rightarrow 0}g_{k}=g\qquad {\rm weakly* in \,\,}L^{\infty
}(0,T;V) ,\bigskip \\ 
\lim\limits_{k\rightarrow 0}\zeta _{k}=\zeta ;\quad %
\lim\limits_{k\rightarrow 0}\lambda _{k}=\lambda \qquad {\rm weakly \,\, in \,\,}%
L^{2}(0,T;L^{2}(\Omega)) .%
\end{array}%
\]%
As far as the function $\w,\g,\zeta $ and $\lambda $ are concerned,
observe that%
\[
\begin{array}{c}
\forall t\in [ t^{n},t^{n+1}] ,\quad \w_{k}(
t) -\u_{k}(t) =\displaystyle\frac{t^{n+1}-t}{k}\left( \u^{n+1}-\u^{n}\right) ,\quad 0\leq n\leq N-1\bigskip \\ 
\forall t\in [ t^{n},t^{n+1}] ,\quad \g_{k}(
t) =\h_{k}(t) =\displaystyle\frac{t^{n+1}-t}{k}\left( \h
^{n+1}-\h^{n}\right) ,\quad 0\leq n\leq N-1\bigskip \\ 
\forall t\in [ t^{n},t^{n+1}] ,\quad \zeta _{k}(t)
-z_{k}(t) =\displaystyle\frac{t^{n+1}-t}{k}\left( z^{n+1}-z^{n}\right)
,\quad 0\leq n\leq N-1\bigskip \\ 
\forall t\in [ t^{n},t^{n+1}] ,\quad \lambda _{k}(t)
-z_{k}(t) =\displaystyle\frac{t^{n+1}-t}{k}\left( z^{n+1}-z^{n}\right)
,\quad 0\leq n\leq N-1.%
\end{array}%
\]%
Therefore%
\begin{equation}
\begin{array}{l}
\left\Vert \w_{k}-\u_{k}\right\Vert _{L^{2}(
0,T,V) }^{2}=\displaystyle\frac{k}{3}\sum_{n=0}^{N-1}\left\Vert \u
^{n+1}-\u^{n}\right\Vert _{H^{1}(\Omega) }^{2},\bigskip
\\ 
\left\Vert \g_{k}-\h_{k}\right\Vert ^{2}=\displaystyle\frac{k}{3}%
\sum_{n=0}^{N-1}\left\Vert \h^{n+1}-\h
^{n}\right\Vert _{H^{1}(\Omega) }^{2},\bigskip \\ 
\left\Vert \zeta _{k}-z_{k}\right\Vert _{L^{2}(\Omega \times ] 0,T%
[) }^{2}=\left\Vert \lambda _{k}-z_{k}\right\Vert _{L^{2}(
\Omega \times ] 0,T[) }^{2}=\displaystyle\frac{k}{3}
\sum_{n=0}^{N-1}\left\Vert z^{n+1}-z^{n}\right\Vert ^{2}.%
\end{array}
\label{E2}
\end{equation}%
Then, using the estimates (\ref{A2}),(\ref{A3}),(\ref{E0}) and the
uniqueness of the limit, we have $\w=\u,\g=\h$ and $\zeta =\lambda =z.$
It remains to prove the strong convergence (\ref{E1}). In view of (\ref{E2}%
), it suffices to prove the strong convergence of $\u_{k}$ and $%
\h_{k}.$ Note that, $(\u_{k}) $ and $(\h_{k}) $
are bouded uniformly in the space%
\[
\left\{ \v\in L^{2}(0,T;H_{0}^{1}(\Omega)
^{2}) ;\frac{\partial \v}{\partial t}\in L^{2}(0,T;L^{4}(
\Omega) ^{2}) \right\} , 
\]%
and as the imbedding of $H^{1}(\Omega) $ into $L^{4}(
\Omega) $ is compact, the Simon's theorem implies that $\u
_{k} $ and $\h_{k}$ converges strongly to $\u$ and $\h
$ respectively in $L^{2}(0,T;L^{4}(\Omega) ^{2}) .$
\end{proof}

\begin{theorem}
Let $\Omega $ be a bounded Lipschitz-continuous domain in two dimensions.
Then for any $\alpha >0,\upsilon >0,\f\in L^{2}(0,T;H( 
{\rm curl};\Omega)) $ and $\u_{0},\h_{0}\in V$
with ${\rm curl}\left( \u_{0}-\alpha \Delta \u_{0}\right)
\in L^{2}(\Omega) ,$ problem (\ref{MM2}) has at least one
solution $\u,\h\in L^{\infty }(0,T;V) $ with $%
\displaystyle\frac{\partial \u}{\partial t},\frac{\partial \h}{\partial t}\in
L^{2}(0,T;V) $ and $p,\omega\in L^{2}(0,T;L_{0}^{2}(
\Omega)) .$
\end{theorem}

\begin{proof}
Let $k$ be a subsequence satisfying the convergences of above Proposition.
It is easy to check that the functions $\u_{k},\h%
_{k},z_{k,}p_{k},\omega_{k},\w_{k},\g_{k},\zeta _{k}$ and $%
\lambda _{k}$ satisfy the following formulations:%
\begin{equation}
\begin{array}{l}
\forall \v\in H_{0}^{1}(\Omega) ,\forall \varphi \in 
\mathcal{C}^{0}([ 0,T]) ,\quad \displaystyle\int_{0}^{T}%
\left[ \left( \frac{\partial }{\partial t}\u_{k}(t),%
\v\right) +\alpha \left( \frac{\partial }{\partial t}\nabla \u
_{k}(t) ,\nabla \v\right) \right. \bigskip  \\ 
\left. +\upsilon \left( \nabla \w_{k}(t) ,\nabla 
\v\right) +\left( \lambda _{k}(t) \times \w_{k}(
t) ,\v\right) -\left( p_{k}(t) ,{\rm div}\v
\right) \right] \varphi (t) dt\bigskip  \\ 
+\displaystyle\int_{0}^{T}\left( \g_{k}(t) \cdot \nabla 
\v,g_{k}(t) \right) \varphi (t)
dt=\displaystyle\int_{0}^{T}\left( \f_{k}(t) ,\v
\right) \varphi (t) dt,%
\end{array}
\label{E3}
\end{equation}%
\[
\]%
\begin{equation}
\begin{array}{l}
\forall \g\in H_{0}^{1}(\Omega) ,\forall \phi \in 
\mathcal{C}^{0}([ 0,T]) ,\quad \displaystyle\int_{0}^{T}%
\left[ \left( \frac{\partial }{\partial t}\h_{k}(t) ,%
\g\right) +\left( \nabla \g_{k}(t) ,\nabla \g \right) \right. \bigskip \\ 
-( \w_{k}(t) \cdot \nabla \g,\
\g_{k}(t)) +(\g_{k}(t) \cdot
\nabla \g,\w_{k}(t)) -( \w_{k},{\rm 
div}\,\g)] \phi (t) dt=0,%
\end{array}
\label{E4}
\end{equation}%
\[
\]%
\begin{equation}
\begin{array}{l}
\forall \theta \in W^{1,4}(\Omega) ,\forall \psi \in \mathcal{C}%
^{1}([ 0,T]) \,\,\,{\rm  with }\,\,\,\psi (T)
=0,\bigskip  \\ 
-\alpha \displaystyle\int_{0}^{T}\left( z_{k}(t) ,\theta \right) 
\frac{\partial }{\partial t}\psi (t) dt+\displaystyle\int_{0}^{T}%
\left[ \upsilon \left( \zeta _{k}(t) ,\theta \right) -\alpha
\left( \w_{k}(t) \cdot \nabla \theta ,\zeta _{k}(
t) \right) \right] \psi (t) dt\bigskip  \\ 
-\left( z^{0},\theta \right) \psi (0) =\displaystyle\int_{0}^{T}
\left[ \upsilon \left( {\rm curl}\w_{k}(t) ,\theta
\right) +\alpha \left( {\rm curl}\f_{k}(t) ,\theta
\right) \right] \psi (t) dt\bigskip  \\ 
+\displaystyle\int_{0}^{T}\alpha \left( {\rm curl}\left( \g_{k}(
t) \cdot \nabla \g_{k}(t) \right) ,\theta \right)
\psi (t) dt,%
\end{array}
\label{E5}
\end{equation}%
where we note that%
\[
\frac{\partial }{\partial t}\u_{k}=\frac{\partial }{\partial t}\left[
\u^{n}+\frac{t-t^{n}}{k}\left( \u^{n+1}-\u
^{n}\right) \right] =\frac{1}{k}\left( \u^{n+1}-\u
^{n}\right) 
\]%
Note that, the weak convergences of the proposition above, imply the
convergences of all the linear terms in (\ref{E3}),(\ref{E4}) and (\ref{E5})
and the terms involving $\f$ also converge, from standard
integration results. Thus, it suffices to check the convergence of the
non-linear terms. Then, for all indices $i$ and $j,$ $1\leq i,j\leq 2,$%
\[
\lim_{k\rightarrow 0}\left( w_{k}\right) _{i}v_{j}\varphi =u_{i}v_{j}\varphi
;\quad \lim_{k\rightarrow 0}\left( g_{k}\right) _{i}g_{j}\phi
=h_{i}g_{j}\phi \quad {\rm strongly \,\, in \,\,} \,L^{2}(\Omega \times ] 0,T%
[ ) 
\]%
and 
\[
\lim_{k\rightarrow 0}\lambda _{k}=z\quad {\rm weakly \,\, in \,\,}\,L^{2}\left( \Omega
\times \left] 0,T\right[ \right) ,
\]%
then, we have%
\[
\begin{array}{l}
\lim\limits_{k\rightarrow 0}\displaystyle\int_{0}^{T}\left( \lambda _{k}(
t) \times \w_{k}(t) ,\v\right) \varphi
(t) dt=\displaystyle\int_{0}^{T}\left( z(t) \times 
\u(t) ,\v\right) \varphi (t)
dt,\bigskip  \\ 
\lim\limits_{k\rightarrow 0}\displaystyle\int_{0}^{T}\left( \g
_{k}(t) \cdot \nabla \v,\g_{k}(t)
\right) \varphi (t) dt=\displaystyle\int_{0}^{T}\left( \h
(t) \cdot \nabla \v,\h(t) \right)
\varphi (t) dt,\bigskip  \\ 
\lim\limits_{k\rightarrow 0}\displaystyle\int_{0}^{T}\left( \w
_{k}(t) \cdot \nabla \g,\g_{k}(t)
\right) \phi (t) dt=\displaystyle\int_{0}^{T}\left( \u
(t) \cdot \nabla \g,\h(t) \right)
\phi (t) dt,\bigskip  \\ 
\lim\limits_{k\rightarrow 0}\displaystyle\int_{0}^{T}\left( \g
_{k}(t) \cdot \nabla \g,\w_{k}(t)
\right) \phi (t) dt=\displaystyle\int_{0}^{T}\left( \h
(t) \cdot \nabla \g,\u(t) \right)
\phi (t) dt.%
\end{array}%
\]%
Similary%
\[
\begin{array}{l}
\lim\limits_{k\rightarrow 0}\left( \w_{k}\cdot \nabla \theta \right)
\psi =\left( \u\cdot \nabla \theta \right) \psi \quad {\rm strongly \,\,
in \,\,}L^{2}(\Omega \times ] 0,T[ ), \bigskip  \\ 
\lim\limits_{k\rightarrow 0}\left( \g_{k}\cdot \nabla {\rm curl}%
\theta \right) \psi =\left( \h\cdot \nabla {\rm curl}\theta \right)
\psi \quad {\rm strongly \,\, in \,\,}L^{2}(\Omega \times ] 0,T[
) 
\end{array}%
\]%
Therefore%
\[
\begin{array}{l}
\lim\limits_{k\rightarrow 0}\displaystyle\int_{0}^{T}\left( \w_{k}\cdot
\nabla \theta ,\zeta _{k}\right) \psi (t)
dt=\displaystyle\int_{0}^{T}\left( \u\cdot \nabla \theta ,z(
t) \right) \psi (t) dt,\bigskip  \\ 
\lim\limits_{k\rightarrow 0}\displaystyle\int_{0}^{T}\left( \g_{k}\cdot
\nabla {\rm curl}\theta ,\g_{k}\right) \psi (t)
dt=\displaystyle\int_{0}^{T}\left( \h\cdot \nabla {\rm curl}\theta ,%
\h\right) \psi (t) dt.%
\end{array}%
\]%
Hence we can pass to the limit in (\ref{E3}), (\ref{E4}) and (\ref{E5}) and
we obtain%
\[
\begin{array}{l}
\forall \v\in H_{0}^{1}(\Omega) ,\forall \varphi \in 
\mathcal{C}^{0}([0,T]) ,\quad \displaystyle\int_{0}^{T}%
\left[ \left( \frac{\partial }{\partial t}\u(t) ,%
\v\right) +\alpha \left( \frac{\partial }{\partial t}\nabla \u
(t) ,\nabla \v\right) \right. \bigskip  \\ 
\left. +\upsilon \left( \nabla \u(t) ,\nabla \v
\right) +\left( z(t) \times \u(t) ,\v
\right) -\left( p(t) ,{\rm div}\v\right) \right]
\varphi (t) dt\bigskip  \\ 
+\displaystyle\int_{0}^{T}\left( \h(t) \cdot \nabla \v
,\h (t) \right) \varphi (t)
dt=\displaystyle\int_{0}^{T}\left( \f(t) ,\v\right)
\varphi (t) dt,%
\end{array}%
\]%
\[
\]%
\[
\begin{array}{l}
\forall \g\in H_{0}^{1}(\Omega) ,\forall \phi \in 
\mathcal{C}^{0}([0,T]) ,\quad \displaystyle\int_{0}^{T}%
\left[ \left( \frac{\partial }{\partial t}\h (t) ,\g
\right) +\left( \nabla \h (t) ,\nabla \g\right)
\right. \bigskip  \\ 
\left. -\left( \u(t) \cdot \nabla \g,\h
(t) \right) +\left( \h (t) \cdot \nabla 
\g,\u (t) \right) -\left( \omega,{\rm div}\g
\right) \right] \phi (t) dt=0,%
\end{array}%
\]%
\[
\]%
\[
\begin{array}{l}
\forall \theta \in W^{1,4}\left( \Omega \right) ,\forall \psi \in \mathcal{C}%
^{1}\left( \left[ 0,T\right] \right) \,\,\,{ with }\,\,\,\psi \left( T\right)
=0,\bigskip  \\ 
-\alpha \displaystyle\int_{0}^{T}\left( z (t) ,\theta \right) \frac{%
\partial }{\partial t}\psi (t) dt+\displaystyle\int_{0}^{T}\left[
\upsilon \left( z (t) ,\theta \right) -\alpha \left( \u
(t) \cdot \nabla \theta ,z(t) \right) \right] \psi
(t) dt\bigskip  \\ 
-\left( z^{0},\theta \right) \psi (0) =\displaystyle\int_{0}^{T}
\left[ \upsilon \left( {\rm curl}\u(t) ,\theta \right)
+\alpha \left( {\rm curl}\f (t) ,\theta \right) \right]
\psi (t) dt\bigskip  \\ 
+\displaystyle\int_{0}^{T}\alpha \left( {\rm curl}\left( \h(
t) \cdot \nabla \h(t) \right) ,\theta \right) \psi
(t) dt.%
\end{array}%
\]%
By choosing $\v,\g\in \mathcal{D}(\Omega) ^{2},\varphi
,\phi $ and $\psi \in \mathcal{D}(] 0,T[) $ and $%
\theta \in \mathcal{D}(\Omega) ,$ we easily recover (\ref{M4}).
It remain to recover the initial data, for this note for any $\g\in
L^{2}(\Omega) ^{2}$ and any $\phi \in H^{1}(0,T) $
satisfying $\phi (T) =0,$ and used the formula $\h
_{k}(t) =\h^{n}+\displaystyle\frac{t-t^{n}}{k}\left( \h
^{n+1}-\h^{n}\right) $ we have%
\[
\displaystyle\int_{0}^{T}\left( \frac{\partial }{\partial t}\h
_{k}(t) ,\g\right) \phi (t)
dt=-\displaystyle\int_{0}^{T}\left( \h_{k}(t) ,\g
\right) \frac{\partial }{\partial t}\phi (t) dt-\left( \h
^{0},\g\right) \varphi (0) 
\]%
Passing to the limit in the equality above, we have%
\[
\int_{0}^{T}\left( \frac{\partial }{\partial t}\h(
t) ,\g\right) \phi (t)
dt=-\displaystyle\int_{0}^{T}\left( \h(t) ,\g
\right) \frac{\partial }{\partial t}\phi (t) dt-\left( \h
^{0},\g\right) \varphi (0) .
\]%
On the other hand, we have%
\[
\int_{0}^{T}( \frac{\partial }{\partial t}\h(
t) ,\g) \phi (t)
dt=-\int_{0}^{T}( \h(t) ,\g
) \frac{\partial }{\partial t}\phi (t) dt- ( \h
(0) ,\g) \varphi (0) 
\]%
where we conclude that $\h^{0}=\h(0) ,$
similarly we obtain $\u^{0}=\u( 0) $ and $%
z^{0}=z(0) .$
\end{proof}

With respect to the uniqueness it is possible to show an analogous to the 
\cite{Girault}. In fact, we have

\begin{theorem}
Assume that $\Omega $ is a convex polygon. Then for any $\alpha >0,\upsilon
>0,$ $\f$ in $L^{2}(0,T;H({\rm curl};\Omega)) $
and $\u_{0}\in V,\h\in V$ with ${\rm curl}( \u
_{0}-\alpha \Delta \u_{0}) \in L^{2}(\Omega) ,$
problem (\ref{M1})-(\ref{MM1}) has exactly one solution $( \u,\h,p, \omega)
\in W.$
\end{theorem}

\end{document}